\documentclass[submission,copyright,creativecommons]{eptcs}
\providecommand{\event}{AiML 2026}

\usepackage{graphicx}

\usepackage{mathUtil}
\usepackage{comment}

\usepackage{quiver}

\usepackage{aiml26}

\usepackage{iftex}

\ifpdf
  \usepackage{underscore}         
  \usepackage[T1]{fontenc}        
\else
  \usepackage{breakurl}           
\fi

\title{Coequivalence Relations and Descent in Modal Logic}
\author{Rodrigo Nicolau Almeida
\institute{ILLC-UvA\\ Amsterdam, the Netherlands}
\email{r.dacruzsilvapinadealmeida@uva.nl}
\and 
Matteo De Berardinis
\institute{ILLC-UvA\\ Amsterdam, the Netherlands\\
Università degli Studi di Salerno\\ Salerno, Italy}
\email{mdeberardinis@unisa.it}}

\newcommand{\titlerunning}{Coequivalence relations and descent in modal logic}
\newcommand{\authorrunning}{R.N. Almeida and M. De Berardinis}

\hypersetup{
  bookmarksnumbered,
  pdftitle    = {\titlerunning},
  pdfauthor   = {\authorrunning},
  pdfsubject  = {EPTCS},               
  pdfkeywords = {modal logic, categorical logic, coequivalence relations, effectiveness,descent morphisms, exactness properties}
}

\begin{document}
\maketitle

\begin{abstract}
A coequivalence relation over a modal logic $L$ is a formula $\rho(\overline{p}_{1},\overline{p}_{2})$ such that the logic $L$ proves that $\rho$ is an equivalence relation. They were introduced by Ghilardi and Zawadowski in the context of the categorical study of non-classical logics. A coequivalence relation is said to  \textit{separate variables} or to be \textit{separating} if it corresponds to a conjunction of formulas of the form $\psi(\overline{p}_{1})\leftrightarrow \psi(\overline{p}_{2})$, which serve as explicit definitions of quotients. A logic $L$ where all coequivalence relations are separating is said to have the \textit{coequivalence separation property} (CoSP). Ghilardi and Zawadowski showed that CoSP fails for $\mathsf{IPC}$. In previous work, the second author showed that such a phenomenon happens already in presumably simpler systems like $\mathsf{S5}$. Ghilardi and Zawadowski therefore raised the question whether a weaker property, formulated in categorical terms related to descent theory, was still true. In this paper, we identify the logical meaning of such a property in relation to CoSP. We introduce the notion of  \emph{local coequivalence relations}, which have the additional structure of a \textit{local transition term}, $\xi(\overline{p},\overline{q}_{1},\overline{q}_{2},\overline{r})$, intuitively capturing the structure of elements lying in the same fiber. We introduce the \emph{local coequivalence separation property} (LCoSP), and prove it to be equivalent, in good cases, to the almost Barr-exactness of $\mathsf{Alg}(L)_{fp}^{op}$. We conclude by showing that $\mathsf{S5}$ has the LCoSP.

\end{abstract}

\section{Introduction}

In algebraic logic it has long been a subject of interest to correspond categorical properties of classes of models, on one hand, to logical properties of calculi, especially definability properties: Beth definability theorems and regularity properties \cite{Maksimova1999} as well as uniform interpolation and the existence of adjoints between free algebras \cite[Chapter 3]{Ghilardi2002},
are two notable examples. In \cite[Chapter 4.9]{Ghilardi2002}, Ghilardi and Zawadowski introduced the concept of a \emph{coequivalence relation}, to serve as a logical correspondent to the notion of an \emph{equivalence relation}\footnote{The precise connection is with $r$-equivalence relations -- which are based on \emph{regular} subobjects, instead of arbitrary subobjects. See Remark \ref{rem: rs and beths} for more context on this.} on the dual of the class of finitely presented algebras. In the simplest case, these are formulas $\rho(\overline{p}_{1},\overline{p}_{2})$ over disjoint tuples of variables $\overline{p}_{1},\overline{p}_{2}$ of the same length, such that the logic $L$ proves the axioms of an equivalence relation for $\rho$. The most basic kind of coequivalence relation is one determined by a conjunction of formulas of the form $\psi(\overline{p}_{1})\leftrightarrow \psi(\overline{p}_{2})$ -- such coequivalence relations are said to \emph{separate variables}. A logic is said to have the \emph{coequivalence separation property} (CoSP) if all of its coequivalence relations separate variables. This stands in analogy with Beth definability, in the following sense:
\begin{enumerate}
    \item \textbf{ Beth definability} says that whenever $\Gamma(\overline{p},q)$ proves $q$ to be constant for any $\overline{p}$ (in the sense that $\Gamma(\overline{p},q_{1}),\Gamma(\overline{p},q_{2})\vdash_{L}q_{1}\leftrightarrow q_{2}$), then there exists an explicit term $\psi(\overline{p})$ which $\Gamma$ proves gives the values of $q$. In a slogan: \emph{implicit functions can be made explicit}.
    \item \textbf{Coequivalence separation} says that whenever a logic $L$ can prove a term $\rho(\overline{p}_1,\overline{p}_2)$ to be an equivalence, there are explicit function terms $\psi_i(\overline{p})$ which compute equivalence classes. In a slogan: \emph{implicit quotients can be made explicit.}
\end{enumerate}

In model theoretic terms, this property is an instance of \textit{uniform elimination of imaginaries} \cite{poizattheoriegalois}. In categorical terms, it amounts to the property that equivalence relations in the category $\mathsf{Alg}(L)_{fp}^{op}$ are \textit{effective}. 

Ghilardi and Zawadowski, using a sheaf-theoretic representation of the category $\mathbf{HA}_{fp}^{op}$, showed that $\mathsf{IPC}$ fails to have the CoSP. Moreover, in \cite{barrcoexactness}, the second author studied coequivalence relations in the context of a natural infinitary extension of modal logic. For locally tabular logics, the results from that framework can be transferred, yielding that normal modal logic $\mathsf{S5}$, much like $\mathsf{IPC}$, has non-separating coequivalence relations.

Ghilardi and Zawadowski  raised the question of whether a weakening of Barr-exactness -- namely the property that all regular epimorphisms are of ``effective descent type" -- holds for $\mathbf{HA}_{fp}^{op}$, and left open what the logical or algebraic interpretation of such a property should be. The property in question is
\emph{almost Barr-exactness}\footnote{This terminology was introduced by Janelidze and Sobral in \cite{Janelidze2009}.} \cite{JanelidzeSobralTholen2003} of the category  $\mathbf{HA}_{fp}^{op}$, or more generally, of $\mathsf{Alg}(L)_{fp}^{op}$ for $L$ a normal modal logic or superintuitionistic logic. In the categorical literature, the property of all regular epimorphisms being of effective descent type has been extensively studied in connection to descent theory \cite{JanelidzeSobralTholen2003}.

 In this paper we introduce the concepts of a \emph{local transition term} $\xi(\overline{p},\overline{q}_{1},\overline{q}_{2},\overline{r})$, and of \emph{local coequivalences} induced by such terms, as well as the \emph{local coequivalence separation property} (LCoSP). Local transition terms take their name from the fact that they encode local compatibility between elements, with local coequivalence relations essentially corresponding to the identification of elements lying inside the orbit of local transitions.  Local coequivalence relations should therefore be thought of as inducing \emph{implicit covering maps}, by analogy with the geometric picture where local homeomorphisms and covering maps come equipped with additional local patching data. Formally, the LCoSP corresponds to all local coequivalences separating variables, and consequently, being given by explicit terms.

In Section \ref{sec: coequivalence relations in logic and category theory}, we give all necessary preliminaries concerning the logical, algebraic and category-theoretic concepts used throughout the paper. In Section \ref{sec: coequivalence relations and eq. relations} we define coequivalence relations and their relationship to categorical equivalence relations. In Section \ref{sec: effective descent and coequivalence} we move to the question of effective descent, and show that it corresponds to LCoSP. In Section \ref{sec: descent equivalence relations in S5 frames} we use a concrete representation of equivalence relations in $\mathsf{S5}$ frames to prove the LCoSP for this system, proving a first of this kind of results.

\section{Logic and Algebra}\label{sec: coequivalence relations in logic and category theory}

Throughout we will be working in the language $\mathcal{L}$ of basic modal logic, namely containing the symbols $(\wedge,\neg,\Box,\bot,\top)$. We work throughout with a set of $\mathcal{L}$-\emph{terms} (also \emph{propositional formulas} or just \emph{formulas}) over a fixed but arbitrary set $\mathsf{Prop}$. We use letter $p,q,r,\dots$ for variables, and write $\overline{p}$ for tuples of distinct variables. Whenever we are given a tuple $\overline{p}$, we will write $\overline{p}_{1},\overline{p}_{2},\dots$ to mean disjoint copies of $\overline{p}$, and for $p\in \overline{p}$, write $p_{i}$ to mean the propositional letter in such a tuple. We use the letters $\phi,\psi,\chi,\dots$ for formulas. The notation $\phi(\overline{p})$ means that the formula $\phi$ has free variables included in the tuple; note that when we write $\phi(\overline{p},\overline{q})$ we always assume all variables to be distinct as they occur in $\phi$.

We assume familiarity with the algebraic semantics of normal modal logics (see e.g. \cite[Chapter 8]{Chagrov1997-cr}). By default a \emph{logic} will mean a normal modal logic, unless otherwise specified. Given a logic $L$, we denote by $\mathsf{Alg}(L)$ the class modal algebras validating the axioms in $L$, also called $L$-algebras. We write $\alpha\vdash_{L}\beta$ to mean that $\beta$ is derivable from $\alpha$ using the axioms in $L$ and the rules of Modus Ponens, Uniform substitution (the latter only applied to axioms), and possibly Necessitation. Given a (finite) set of generators $X$, we write $\mathcal{F}_{L}(X)$ to mean the free $X$-generated $L$-algebra. We write $\mathsf{Alg}(L)_{fp}$ for the full subcategory of $\mathsf{Alg}(L)$ given by \emph{finitely presented} $L$-algebras (see \cite[Chapter 2]{Ghilardi2002} for facts about finitely presented $L$-algebras and their duals). These can be identified with quotients of the form $\mathcal{F}_{L}(\overline{p})/\alpha(\overline{p})$, where $\alpha(\overline{p})$ is an arbitrary formula; the quotient is constructed by taking the quotient under $\sim_{\alpha}$, given by, for $\phi(\overline{p}),\psi(\overline{p})$ two formulas:
\begin{equation*}
    \phi\sim_{\alpha}\psi \iff \alpha(\overline{p})\vdash_{L}\phi(\overline{p})\leftrightarrow\psi(\overline{p}).
\end{equation*}
A morphism of finitely presented $L$-algebras $\mathcal{F}_L(\overline{p})/\alpha(\overline{p}) \to \mathcal{F}_L(\overline{q})/\beta(\overline{q})$ is specified by a tuple of formulas $\overline{\theta} = (\theta_p(\overline{q}) \colon p \in \overline{p})$ in the variables $\overline{q}$ such that $\beta(\overline{q}) \vdash_L \alpha(\overline{\theta}/\overline{p})$ (we denote by $\alpha(\overline{\theta}/\overline{p})$, or simply by $\alpha(\overline{\theta})$, the substitution of $\theta_p(\overline{q})$ in every instance of $p$ in $\alpha$). If $\overline{\theta}'$ is another such tuple of formulas, then $\overline{\theta}$ and $\overline{\theta}'$ specify the same morphism if and only if $\beta(\overline{q}) \vdash_L \theta_p(\overline{q}) \leftrightarrow \theta'_p(\overline{q})$ for all $p \in \overline{p}$. If we fix the domain $\mathcal{F}_L(\overline{p})/\alpha(\overline{p})$, any morphism of finitely presented $L$-algebras out of $\mathcal{F}_L(\overline{p})/\alpha(\overline{p})$ can be seen as an extension of the theory $\alpha(\overline{p})$, both in the propositional letters and in the relations imposed; in fact, each one of them can be presented as
$$\mathcal{F}_{L}(\overline{p})/\alpha(\overline{p})\to  \mathcal{F}_{L}(\overline{p},\overline{q})/\beta(\overline{p},\overline{q})$$
with
$\beta(\overline{p},\overline{q})\vdash_L \alpha(\overline{p})$ (the substitution being $\theta_p(\overline{p},\overline{q}) = p$). Among these, the quotients (regular epimorphisms, in the language of category theory) are those extensions for which no new propositional variable is added; in other words, they are given by $\beta(\overline{p})$ such that $\beta(\overline{p}) \vdash_L \alpha(\overline{p})$.

We write $\mathsf{Alg}(L)_{fp}^{op}$ for the dual category of $\mathsf{Alg}(L)_{fp}$. This category will be identified via duality, when convenient, with a category of relational spaces and frames.

\begin{notation}\label{not: conventions}
    When proving the correspondence between syntactic and categorical properties, we will use the following convenient notation. We identify a formula $\alpha(\overline{p})$ with the object in $\mathsf{Alg}(L)_{fp}^{op}$ dual to the finitely presented algebra $\mathcal{F}_L(\overline{p})/\alpha(\overline{p})$. Moreover, given a substitution $\overline{\theta} = (\theta_p(\overline{q}) \colon p \in \overline{p})$ such that $\beta(\overline{q}) \vdash_L \alpha(\overline{\theta}/\overline{p})$, we let the derivation $\beta(\overline{q}) \vdash_L \alpha(\overline{\theta}/\overline{p})$ denote the morphism in $\mathsf{Alg}(L)_{fp}^{op}$ dual to the morphism of finitely presented algebras $\mathcal{F}_L(\overline{p})/\alpha(\overline{p}) \to \mathcal{F}_L(\overline{q})/\beta(\overline{q})$ specified by $\overline{\theta}$. Observe that compositions are obtained by composing substitutions: the composite of $\gamma(\overline{r}) \vdash_L \beta(\overline{\mu}/\overline{q})$ and $\beta(\overline{q}) \vdash_L \alpha(\overline{\theta}/\overline{p})$ is $\gamma(\overline{r}) \vdash_L \alpha(\overline{\theta}(\overline{\mu}/\overline{q})/\overline{p})$.
\end{notation}

In many cases, we will assume that $\mathsf{Alg}(L)_{fp}^{op}$ is a \textit{regular category}. This means a category with finite limits, pullback-stable regular epimorphisms, admitting coequalizers of kernel pairs (see e.g. \cite[Volume 2, Chapter 4]{Borceux1994}). In logical terms, the condition of regularity can be reformulated as follows.

\begin{definition}[Beth definability]
We say that a logic $L$ has the \emph{Beth definability property} if for any finite sets $\overline{p},q$ of propositional variables, and formula $\phi(\overline{p},q)$, whenever $\phi(\overline{p},q_1) \wedge \phi(\overline{p},q_2) \vdash_L q_1 \leftrightarrow q_2$, then there exists a formula $\psi(\overline{p})$ such that $\phi(\overline{p},q) \vdash_L q \leftrightarrow \psi(\overline{p})$.
\end{definition}

\begin{definition}[$\exists_{\mathrm{MIP}}$]
We say that a logic $L$ has the \emph{right-uniform Maehara interpolation property} ($\exists_{\mathrm{MIP}}$) if for any finite sets $\overline{p},\overline{q}$ of propositional variables, and formula $\phi(\overline{p},\overline{q})$, there is a formula $\mu(\overline{q})$ such that for any formulas
$\psi(\overline{q},\overline{r})$ and $\chi(\overline{q},\overline{r})$:
\begin{equation*}
\phi(\overline{p},\overline{q})\wedge \psi(\overline{q},\overline{r})\vdash_{L}\chi(\overline{q},\overline{r}) \iff \mu(\overline{q})\wedge \psi(\overline{q},\overline{r})\vdash_{L}\chi(\overline{q},\overline{r})
\end{equation*}
\end{definition}

\begin{remark}\label{rem: rs and beths}
   The Beth property ensures that the \textit{r}- version of regularity coincides with the standard one (as explained in \cite[Chapter 2.3]{Ghilardi2002}). Throguhout the paper we will assume our logic $L$ to have the Beth definability property; for this reason we will not make any distinction between r-regularity and regularity, or between r-equivalence relations and equivalence relations.
\end{remark}

The property $\exists_{\mathrm{MIP}}$ entails regularity of $\mathsf{Alg}(L)_{fp}^{op}$: see e.g. \cite[Appendix B]{almeidaghilardiunificationsimplevariable}. We then have the following:

\begin{proposition}\label{prop: regularity of the category in question}
    If $L$ is a logic with the Beth property and with $\exists_{\mathrm{MIP}}$, then  $\mathsf{Alg}(L)_{fp}^{op}$ is a regular category.
\end{proposition}

\section{Coequivalence relations and equivalence relations}\label{sec: coequivalence relations and eq. relations}

We move on to the key definition of a coequivalence relation. The following facts are noted in \cite[Chapter 4.9]{Ghilardi2002} and \cite[Section 5.4]{deberardinisprooftheoryforprofinite}. We fix throughout a given normal modal logic $L$ with the \textit{Beth definability property}.

\begin{definition}[Coequivalence relation]\label{def:eqrelalg}
    Let $\rho(\overline{p}_{1},\overline{p}_{2})$ be a formula over tuples of the same length $\overline{p}_{1},\overline{p}_{2}$. Given a formula $\tau(\overline{p})$ we say that $\rho$ is a $\tau$-\emph{coequivalence relation} if $\rho(\overline{p}_{1},\overline{p}_{2})\vdash_L\tau(\overline{p}_{1})\wedge \tau(\overline{p}_{2})$, and it satisfies:
    \begin{enumerate}
        \item (Reflexivity) $\tau(\overline{p})\vdash_{L} \rho(\overline{p},\overline{p})$;
        \item (Symmetry) $\rho(\overline{p}_{1},\overline{p}_{2})\vdash_{L}\rho(\overline{p}_{2},\overline{p}_{1})$;
        \item (Transitivity) $\rho(\overline{p}_{1},\overline{p}_{2})\wedge\rho(\overline{p}_{2},\overline{p}_{3})\vdash_{L}\rho(\overline{p}_{1},\overline{p}_{3})$.
    \end{enumerate}
\end{definition}

As a simple example, notice that for a tuple $\psi_{1}(\overline{p}),...,\psi_{m}(\overline{p})$ of formulas, the formula $\bigwedge_{i=1}^{m}\psi_{i}(\overline{p}_{1})\leftrightarrow \psi_{i}(\overline{p}_{2})$ is a coequivalence relation.

Coequivalence relations coincide precisely with equivalence relations, in a categorical sense, as we proceed to explain.

\begin{notation}
    Throughout, given a category $\mathbf{C}$ with finite limits, and objects $X,Y,Z\in \mathbf{C}$:
    \begin{enumerate}
        \item We write $X\otimes Y$ to mean the categorical product;
        \item Given morphisms $f\colon X\to Z$ and $g\colon Y\to Z$, we write $X\otimes_{Z}Y$ to mean the categorical pullback.
    \end{enumerate}
\end{notation}

\begin{definition}[Equivalence relation]\label{def:eqrel}
    Let $\mathbf{C}$ be a category with all finite limits, and $X$ an object there. An \textit{equivalence relation} over $X$ is a subobject $(\pi_1,\pi_2) \colon R\hookrightarrow X\otimes X$, 
    equipped with the following morphisms:
    \begin{enumerate}
        \item \textit{internal reflexivity}, $r \colon X\to R$, which is a section of both $\pi_{i}$, i.e., $\pi_{1}r=\pi_{2}r=1_{X}$;
        \item \textit{internal symmetry}, $s \colon R\to R$, such that $\pi_{1}s=\pi_{2}$ and $\pi_{2}s=\pi_{1}$;
        \item \textit{internal transitivity}, $t \colon R\otimes_{X}R\to R$, such that $\pi_1 t = \pi_1 \kappa_1$ and $\pi_2 t = \pi_2 \kappa_2$, where $R\otimes_{X}R$ is the pullback
\[\begin{tikzcd}
	{R \otimes_X R} & R \\
	R & X
	\arrow["{\kappa_1}", from=1-1, to=1-2]
	\arrow["{\kappa_2}"', from=1-1, to=2-1]
	\arrow["{\pi_2}", from=1-2, to=2-2]
	\arrow["{\pi_1}"', from=2-1, to=2-2]
\end{tikzcd}\]
\end{enumerate}
\end{definition}

\begin{definition}[Effective equivalence]\label{def:effeqrel}
    Given a category $\mathbf{C}$ with finite limits, and an object $X$, we say that an equivalence relation $R$ over $X$ is \emph{effective} if there exists a morphism $f\colon X\to Q$ such that $R$ is (isomorphic to) the pullback $X\otimes_{Q}X$, called the \emph{kernel pair} of $f$.
\end{definition}

\begin{remark}\label{rem:eqrel}
The categorical product $X \otimes X$ is always an effective equivalence relation over $X$; the maps $r$, $s$ and $t$, witnessing the internal properties of this equivalence relation, are induced by the universal property of the product. One can therefore alternatively say that a subobject $R$ of $X \otimes X$ satisfies:
\begin{enumerate}
    \item internal reflexivity if and only if $r \colon X\to X \otimes X$ factorizes through $R$;
    \item internal symmetry if and only if the restriction of $s \colon X \otimes X \to X \otimes X$ to $R$ factorizes through $R$;
    \item internal transitivity if and only if the restriction of $t \colon (X \otimes X) \otimes_{X} (X \otimes X) \to X \otimes X$ to $R\otimes_{X}R$ (which is a subobject of $(X \otimes X) \otimes_{X} (X \otimes X)$) factorizes through $R$.
\end{enumerate}
\end{remark}

The effectiveness of equivalence relations is often considered for regular categories. Such a combination gives rise to the following important notion:

\begin{definition}[Barr-exact category]\label{def: Barr-exactness in categories}
    We say that a category is \emph{Barr-exact} if it is regular and all equivalence relations are effective.
\end{definition}

We briefly sketch the correspondence between coequivalence relations and equivalence relations in $\mathsf{Alg}(L)_{fp}^{op}$, as outlined in \cite[Chapter 4.8]{Ghilardi2002}. Recall our convention from \ref{not: conventions}. The product in $\mathsf{Alg}(L)_{fp}^{op}$ of $\tau(\overline{p})$ with itself is given by $\tau(\overline{p}_1) \wedge \tau(\overline{p}_2)$, with projections $\tau(\overline{p}_1) \wedge \tau(\overline{p}_2) \vdash_L \tau(\overline{p}_1/\overline{p})$ and $\tau(\overline{p}_1) \wedge \tau(\overline{p}_2) \vdash_L \tau(\overline{p}_2/\overline{p})$.

A subobject of $\tau(\overline{p}_1) \wedge \tau(\overline{p}_2)$ is given by $\rho(\overline{p}_1,\overline{p}_2)$ such that $\rho(\overline{p}_{1},\overline{p}_{2}) \vdash_L \tau(\overline{p}_{1})\wedge\tau(\overline{p}_{2})$ (as noted in \ref{not: conventions}, quotients in $\mathsf{Alg}(L)_{fp}$ are given by those extensions of theories for which no new propositional variable is added).

\begin{proposition}\label{prop: coequivalence to dual}
    A finitely presented algebra $\mathcal{F}_{L}(\overline{p}_{1},\overline{p}_{2})/\rho(\overline{p}_{1},\overline{p}_{2})$ is dual to an equivalence relation over $\tau(\overline{p})$ in $\mathsf{Alg}(L)_{fp}^{op}$ if and only if $\rho(\overline{p}_{1},\overline{p}_{2})$ is a $\tau$-coequivalence relation. More specifically:
    \begin{enumerate}
        \item an internal reflexivity map 
        exists if and only if $\tau(\overline{p})\vdash_{L}\rho(\overline{p},\overline{p})$;
        \item an internal symmetry map 
        exists if and only if $\rho(\overline{p}_{1},\overline{p}_{2})\vdash_{L}\rho(\overline{p}_{2},\overline{p}_{1})$;
        \item an internal transitivity map $t$ exists if and only if $\rho(\overline{p}_{1},\overline{p}_{2})\wedge\rho(\overline{p}_{2},\overline{p}_{3})\vdash_{L}\rho(\overline{p}_{1},\overline{p}_{3})$.
    \end{enumerate}
\end{proposition}
\begin{proof}
We prove 1., which illustrates the arguments for the other cases. First recall the restricted projections, $$\rho(\overline{p}_1,\overline{p}_2) \vdash_L \tau(\overline{p}_1) \wedge \tau(\overline{p}_2) \vdash_L \tau(\overline{p}_1/\overline{p}) \quad \text{ and } \quad \rho(\overline{p}_1,\overline{p}_2) \vdash_L \tau(\overline{p}_1) \wedge \tau(\overline{p}_2) \vdash_L \tau(\overline{p}_2/\overline{p}).$$

Internal reflexivity is given by a map $\tau(\overline{p})\vdash_{L}\rho(\overline{\mu}_1/\overline{p}_1,\overline{\mu}_2/\overline{p}_2)$, whose compositions with such restricted projections, namely: $$\tau(\overline{p}) \vdash_L \tau(\overline{\mu}_1/\overline{p}_1) \quad \text{ and } \quad \tau(\overline{p}) \vdash_L \tau(\overline{\mu}_2/\overline{p}_2),$$ are the identity. This means that we can equivalently rewrite the map witnessing internal reflexivity simply as $\tau(\overline{p})\vdash_{L}\rho(\overline{p}/\overline{p}_1,\overline{p}/\overline{p}_2)$, i.e., $\tau(\overline{p})\vdash_{L}\rho(\overline{p},\overline{p})$.
\end{proof}

\begin{definition}[Coeq. separating variables]
    A $\tau$-coequivalence relation $\rho(\overline{p}_{1},\overline{p}_{2})$ is said to \emph{separate variables} or simply be \emph{separating} if there exists a sequence of formulas $\psi_{1}(\overline{p}),\dots,\psi_{m}(\overline{p})$, such that
        \begin{equation*}
        \rho(\overline{p}_{1},\overline{p}_{2}) \dashv\vdash_{L} \tau(\overline{p}_{1})\wedge \tau(\overline{p}_{2})\wedge \bigwedge_{i=1}^{m}\psi_{i}(\overline{p}_{1})\leftrightarrow\psi_{i}(\overline{p}_{2}).
    \end{equation*}
\end{definition}

The following was first observed by Ghilardi and Zawadowski \cite[pp.101]{Ghilardi2002}:

\begin{proposition}\label{prop: effective to coequivalence}
A finitely presented algebra $\mathcal{F}_{L}(\overline{p}_{1},\overline{p}_{2})/\rho(\overline{p}_{1},\overline{p}_{2})$ is dual to an effective equivalence relation over $\tau(\overline{p})$ in $\mathsf{Alg}(L)_{fp}^{op}$ if and only if $\rho(\overline{p}_{1},\overline{p}_{2})$ is a separating $\tau$-coequivalence relation.
\end{proposition}
\begin{proof}
It is sufficient to observe that the formula in the definition above represents the kernel pair of a morphism $\tau(\overline{p}) \vdash_L \sigma(\overline{\psi}/\overline{r})$ in $\mathsf{Alg}(L)_{fp}^{op}$.
\end{proof}

\begin{definition}[CoSP]
    A logic $L$ is said to have the \emph{coequivalence separation property} (CoSP) if for each formula $\tau(\overline{p})$, each $\tau$-coequivalence relation is separating.
\end{definition}

From Proposition \ref{prop: coequivalence to dual} and Proposition \ref{prop: effective to coequivalence}, the following is the categorical meaning of the CoSP:

\begin{proposition}\label{prop:CoSP iff eff}
    A logic $L$ with the Beth property has the CoSP if and only if in $\mathsf{Alg}(L)_{fp}^{op}$ all equivalence relations are effective.
\end{proposition}

\begin{remark}
    In model theoretic terms, the CoSP can be related to the framework of \textit{ imaginary elements}  \cite{poizattheoriegalois}. This was introduced by Bruno Poizat as a generalization of classical Galois theory, with the intention of providing a model-theoretic framework which generalizes  the fundamental theorem of Galois theory. Our coequivalence relations are precisely equivalence formulas. The existence of formulas $\theta(\overline{x})$ witnessing the separation corresponds to the existence of a uniform schema for eliminating imaginaries -- that is, picking equivalence classes in algebraic models. It would be interesting, in this vein, to determine the precise connection between uniform elimination of imaginaries and the CoSP, in a similar way to the known connections between interpolation and amalgamation (e.g., \cite{metcalfeamalgamation}), or uniform interpolation and model completions (e.g., \cite{VANGOOLmetcalfetsinakisuniforminterpolation}). 
\end{remark}

We now give a few examples to illustrate the meaning of the CoSP:

\begin{example}[\textsf{CPC} has the CoSP]\label{ex: cpc cosp proof}
    We will show that every coequivalence relation over $\mathsf{CPC}$ is separating, through elementary methods \footnote{Notice that by Proposition \ref{prop:CoSP iff eff}, it would suffice to observe that in $\mathbf{BA}_{fp}^{op}$ all equivalence relations are effective, which is well-known since such a category is equivalent to $\mathbf{FinSet}$, and the latter category is a topos, where all equivalence relations are always effective.}. Assume that $\rho(\overline{p}_{1},\overline{p}_{2})$ is a coequivalence relation (over $\top$). The conditions of being a coequivalence relation means three things:
    \begin{itemize}
        \item (reflexivity) for any truth-assignment  $v\colon \overline{p}_{1},\overline{p}_{2}\to \{0,1\}$ where for each $p\in \overline{p}$, $v(p_{1})=v(p_{2})$, then $v$ makes $\rho(\overline{p}_{1},\overline{p}_{2})$ true;
        \item (symmetry) if $v$ makes $\rho(\overline{p}_{1},\overline{p}_{2})$ true, then if $v^{\dagger}$ is defined by setting, for each $p\in \overline{p}$, $v^{\dagger}(p_{1})=v(p_{2})$ and $v^{\dagger}(p_{2})=v(p_{1})$, we have that $v^{\dagger}$ also makes $\rho(\overline{p}_{1},\overline{p}_{2})$ true;
        \item (transitivity) if $v$ 
    makes $\rho(\overline{p}_{1},\overline{p}_{2})$ true and $v' \colon \overline{p}_{2},\overline{p}_{3} \to \{0,1\}$ is another valuation which makes $\rho(\overline{p}_{2},\overline{p}_{3})$ true, and such that for each $p\in \overline{p}$, $v(p_{2})=v'(p_{2})$, 
    then the valuation $v''=v\cup v'$
    makes $\rho(\overline{p}_{1},\overline{p}_{3})$ true.
    \end{itemize} 

    Given such a coequivalence relation, define an equivalence relation on the set $\mathsf{Val}(\overline{p})$ of valuations $v$ over $\overline{p}$ as follows: $v\sim v'$ if and only if there is a valuation $v_{0}$ over $\overline{p}_{1},\overline{p}_{2}$ such that $v_{0}{\restriction}_{\overline{p}_{1}}=v$ and $v_{0}{\restriction}_{\overline{p}_{2}}=v'$ and $v_{0}$ validates $\rho(\overline{p}_{1},\overline{p}_{2})$. This is an equivalence relation by the properties just listed. For each equivalence class $C_{i}$, let $\psi_{i}$ be the formula given by the disjunction of the (atoms corresponding to the) valuations in $C_{i}$, so that a valuation $v$ over $\overline{p}$ validates $\psi_i$ if and only if $v \in C_i$. Then we claim that:
    \begin{equation*}
         \rho(\overline{p}_{1},\overline{p}_{2})\dashv\vdash \bigwedge_{i=1}^{m}\psi_{i}(\overline{p}_{1})\leftrightarrow \psi_{i}(\overline{p}_{2}).
    \end{equation*}

    Notice that a valuation $v$ on $\overline{p}_1$ and $\overline{p}_2$ validating the conjunction means that ``$v{\restriction}_{\overline{p}_{1}} \in C_i$ if and only if $v{\restriction}_{\overline{p}_{2}} \in C_i$'' for all $i$; this is equivalent to say that $v{\restriction}_{\overline{p}_{1}} \sim v{\restriction}_{\overline{p}_{2}}$, since $v{\restriction}_{\overline{p}_{1}}$ must belong to some $C_i$. Then, 
    by construction any valuation that validates $\rho$ validates the conjunction. Vice versa, if $v$ is a valuation validating the conjunction, 
    then $v{\restriction}_{\overline{p}_{1}} \sim v{\restriction}_{\overline{p}_{2}}$. This means by assumption that there is some $v_{0}$ satisfying $\rho$, such that 
    $v_{0}{\restriction}_{\overline{p}_{1}}=v{\restriction}_{\overline{p}_{1}}$ and $v_{0}{\restriction}_{\overline{p}_{2}}=v{\restriction}_{\overline{p}_{2}}$. 
    This forces $v = v_0$, hence $v$ validates $\rho$.
\end{example}

It will be useful to contrast this picture with the setting of $\mathsf{S5}$. For that purpose we recall the frames for $\mathsf{S5}$.

\begin{definition}[S5-frames]\label{def: s5-frame}
    A pair $(X,R)$ is called an $\mathsf{S5}$\textit{-frame} if $R$ is an equivalence relation over a set $X$. We say that the frame is \textit{locally finite} if for each $x\in X$, $R[x]=\{y\in X : xRy\}$ is finite. It is \textit{finite} if $X$ is finite. Obviously, finite frames are locally finite. Given such a frame, a subset $C\subseteq X$ is called a \textit{cluster} if it is an $R$-equivalence class.
    We say that $X$ is a \textit{simple} frame if $X$ itself is a cluster.
\end{definition}

\begin{definition}[S5-model]
    An $\mathsf{S5}$-\emph{model} over $\overline{p}$ is a pair $\mathfrak{M},x$, where $\mathfrak{M}$ is a pair consisting of an $\mathsf{S5}$-frame $(X,R)$, together with a function $X \to \mathsf{Val}(\overline{p})$ (we denote by $v_y \colon \overline{p} \to \{0,1\}$ the valuation associated to $y \in X$), and $x \in X$. If $\overline{q} \subseteq \overline{p}$, then $\mathfrak{M}^{\overline{q}},x$ denotes the obvious restriction of $\mathfrak{M},x$.
\end{definition}

We assume the reader is familiar with the usual Kripke semantics for modal logic. Recall also that by Kripke completeness of $\mathsf{S5}$ with respect to finite simple frames, we have that for any pair of formulas $\phi,\psi$:
\begin{equation*}
    \phi\vdash_{\mathsf{S5}}\psi \iff \forall  \mathfrak{M} \text{ based on a finite cluster}, (\forall x\in \mathfrak{M}, \mathfrak{M},x\Vdash \phi) \Rightarrow (\forall x\in \mathfrak{M}, \mathfrak{M},x\Vdash \phi). 
\end{equation*}

Recall the following representation of modal formulas: 

\begin{definition}[Nabla form]
    Let $\chi_{1},\dots,\chi_{n}$ be a collection of formulas. We denote by $\nabla(\chi_{1},\dots,\chi_{n})$, the \textit{nabla form}, the formula $\bigwedge_{i=1}^{n}\Diamond\chi_{i}\wedge \Box\left(\bigvee_{i=1}^{n}\chi_{i}\right)$.
\end{definition}

\begin{example}[\textsf{S5} contains non-separating coequivalences]\label{ex: s5 non separating}
    Let $p$ be a propositional variable; we call a valuation $v$ over $p_{1},p_{2}$ \textit{matching}, or $(p_{1},p_{2})$-matching, if $v(p_{1})=v(p_{2})$. Denote by $\mathsf{Val}(p_{1},p_{2})$ the set of valuations, and by $\mathsf{Mat}(p_{1},p_{2})$ the set of matching valuations over these variables.   
    Then consider the formula:
    \begin{equation*}
        \rho(p_{1},p_{2})=\bigvee_{S\subseteq \mathsf{Mat}(p_1,p_2)}\nabla(v \colon v \in S) \vee \bigvee_{T\subseteq \mathsf{Val}(p_{1},p_{2})\setminus \mathsf{Mat}(p_{1},p_{2})}\nabla(v \colon v\in T),
    \end{equation*}
    Then $\rho$ is a coequivalence relation over $\mathsf{S5}$. To see this, note that $\rho(p,p)$ is logically equivalent to the disjunction of all possible (irreducible) simple \textsf{S5}-models over the variable $p$, hence it is a tautology. 
    Moreover, $\rho(p_{1},p_{2})$ is logically equivalent to $\rho(p_{2},p_{1})$. Finally assume that $\mathfrak{M},x$ is a model over $p_{1},p_{2},p_{3}$, that validates $\rho(p_{1},p_{2})$ and $\rho(p_{2},p_{3})$. Let $\mathfrak{M}^{p_{1},p_{2}},x$ and $\mathfrak{M}^{p_{2},p_{3}},x$ be the two restrictions of this model to the specified propositional letters. By considering the four cases, we see that $\mathfrak{M}^{p_{1},p_{3}}$ will be matching whenever both pairs match or both do not match, and will be non-matching otherwise. Thus transitivity follows.
   
    Now we show that this coequivalence is not separating. Assume that there is a sequence of formulas $\psi_{1}(p),\dots,\psi_{m}(p)$ such that $\rho\vdash \bigwedge_{i=1}^{m}\psi_{i}(p_{1})\leftrightarrow \psi_{i}(p_{2})$. We prove that $\bigwedge_{i=1}^{m}\psi_{i}(p_{1})\leftrightarrow \psi_{i}(p_{2}) \vdash \rho$ cannot be the case. Let $\mathfrak{M}$ be a model over a (finite) cluster, containing exactly one point $y \in \mathfrak{M}$ such that $v_y = v$, for each $v\notin \mathsf{Mat}(p_{1},p_{2})$. By assumption, $\mathfrak{M}\Vdash \rho(p_{1},p_{2})$, so $\mathfrak{M}\Vdash \bigwedge_{i=1}^{m}\psi_{i}(p_{1})\leftrightarrow\psi_{i}(p_{2})$. Let $\mathfrak{M}^{+}$ be the model obtained by adding to the cluster $\mathfrak{M}$ a point $e$ corresponding to a single fixed valuation $v$ which is matching. We claim that then it is still the case that $\mathfrak{M}^{+}\Vdash \bigwedge_{i=1}^{m}\psi_{i}(p_{1})\leftrightarrow \psi_{i}(p_{2})$.
    To see this, observe that, for each $x \in \mathfrak{M}$ (i.e.\ $x \in \mathfrak{M}^+$ different from $e$), $(\mathfrak{M}^+)^{p_1},x \cong \mathfrak{M}^{p_1},x$: this is because $(\mathfrak{M}^+)^{p_1},x$ admits a surjective p-morphism to $\mathfrak{M}^{p_1},x$ ($y \in \mathfrak{M}^+$ different from $e$ is sent to itself, while $e$ is sent to a non-matching point $e^*$, whose valuation restricted to $p_1$ coincides with that of $e$). Similarly $(\mathfrak{M}^{+})^{p_{2}},x\cong \mathfrak{M}^{p_{2}},x$. Thus we have that for each $\psi_{i}$, and $x\in \mathfrak{M}$, by composing these equivalences with the assumption that $\mathfrak{M},x\Vdash \psi_{i}(p_{1})\leftrightarrow \psi_{i}(p_{2})$:    \begin{align*}
        \mathfrak{M}^{+},x\Vdash \psi_{i}(p_{1}) \iff  \mathfrak{M}^+,x\Vdash \psi_{i}(p_{2}). &
    \end{align*}
    Moreover, $(\mathfrak{M}^+)^{p = p_1},e \cong (\mathfrak{M}^+)^{p = p_2},e$ (via the isomorphism sending $e$ to itself and each $x \in \mathfrak{M}$ to its symmetric $x^{\dagger}$ exchanging the valuations of $p_1$ and $p_2$). Reasoning as before, we have that $\mathfrak{M}^+,e \Vdash \psi_i(p_1)$ if and only if $\mathfrak{M}^+,e \Vdash \psi_i(p_2)$.
    We conclude that $\mathfrak{M}^{+}\Vdash \bigwedge_{i=1}^{m}\psi_{i}(p_{1})\leftrightarrow\psi_{i}(p_{2})$.

    However, $\mathfrak{M}^{+},e\nVdash \rho(p_{1},p_{2})$, since the point $e$ neither sees only matching valuations nor does it see only valuations which do not match, by construction. Thus we have that $\rho$ does not separate variables. 
\end{example}

\begin{remark}\label{rem: sheaf-like condition}
From Example \ref{ex: s5 non separating}, we see that the CoSP -- at least in the case of locally tabular logics like $\mathsf{S5}$ -- is related to a sheaf-like condition, similar to the uniqueness of gluing. The reason why the CoSP holds in \textsf{CPC} is the obvious fact that a valuation $v$ over $\overline{p}_1$ and $\overline{p}_2$ is uniquely determined by its restrictions $v{\restriction}_{\overline{p}_{1}}$ and $v{\restriction}_{\overline{p}_{2}}$. This fails in \textsf{S5}: there can be non-bisimilar models $\mathfrak{M},x$ and $\mathfrak{N},y$ over $\overline{p}_1$ and $\overline{p}_2$ such that $\mathfrak{M}^{\overline{p}_1},x \cong \mathfrak{N}^{\overline{p}_1},y$ and $\mathfrak{M}^{\overline{p}_2},x \cong \mathfrak{N}^{\overline{p}_2},y$. More so, as observed in \cite[Theorem 6.10]{barrcoexactness}, in the extension $\mathsf{S5}_{2}$ of $\mathsf{S5}$ where clusters are restricted to at most $2$-elements, the property again holds for similar reasons.
\end{remark}

As remarked in the introduction, Ghilardi and Zawadowski showed that $\mathsf{IPC}$ fails to have the CoSP, and subsequently asked whether a natural weakening of it -- categorically corresponding to \emph{almost Barr-exactness} -- actually holds. In light of Example \ref{ex: s5 non separating}, such a question can be asked already for $\mathsf{S5}$. Moreover, they left open what the logical and algebraic content of such a property should be. We turn to this in the next section.

\section{Local coequivalence relations and effective descent}\label{sec: effective descent and coequivalence}

\subsection{Local transition terms}

In the previous section we saw that coequivalence relations encode in a logical formula some implicit identification of elements. This is a general phenomenon in mathematics: in topology, for instance, equivalence relations on a space likewise provide implicit identifications between elements. Nevertheless, there might be equivalences  which identify elements in terms of some algebraic data -- such as what happens in the setting of topology with local homeomorphisms, and sheaves over a space.

In a logical setting, one can likewise ask  whether such local kinds of equivalences appear. For this purpose, we introduce the main technical tool arising in the context of our work. For ease of notation, given $\alpha(\overline{p})$, we write $\mathsf{Ext}(\alpha)=\{\beta(\overline{p},\overline{q}) : \beta\vdash_L\alpha\}$.

\begin{definition}\label{def:transition}
    Let $\gamma(\overline{p},\overline{q},\overline{r})\in \mathsf{Ext}(\beta(\overline{p},\overline{q}))$. A sequence $\overline{\xi}$ of formulas over the variables $\overline{p},\overline{q}_{1},\overline{q}_{2},\overline{r}$, namely $\xi_{r}(\overline{p},\overline{q}_{1},\overline{q}_{2},\overline{r})$ for each $r \in \overline{r}$, is called a sequence of \textit{local transition terms} for $\gamma$ if the following hold:
    \begin{enumerate}
        \item ($\gamma$-compatibility):  $\gamma(\overline{p},\overline{q}_{1},\overline{r})\wedge\beta(\overline{p},\overline{q}_{2})\vdash_{L}\gamma(\overline{p},\overline{q}_{2},\overline{\xi})$.
        \item (Identity): For each $r \in \overline{r}$, $\gamma(\overline{p},\overline{q},\overline{r})\vdash_{L} r \leftrightarrow \xi_{r}(\overline{p},\overline{q},\overline{q},\overline{r})$.
        \item (Cocycle):  For each $r \in \overline{r}$,  $\gamma(\overline{p},\overline{q}_{1},\overline{r})\wedge \beta(\overline{p},\overline{q}_{2})\wedge \beta(\overline{p},\overline{q}_{3})\vdash_{L} \xi_{r}(\overline{p},\overline{q}_{1},\overline{q}_{3},\overline{r})\leftrightarrow \xi_{r}(\overline{p},\overline{q}_{2},\overline{q}_{3},\overline{\xi})$.
    \end{enumerate}
\end{definition}

To obtain some intuition for these conditions, assume that $\overline{p}=\emptyset$, $\overline{q}_{1}=\{q_{1}\}$, $\overline{q}_{2}=\{q_{2}\}$ and $\overline{r}=\{r\}$. Then one can think of $\xi$ as defining, for every pair of terms $b,b'$, a function $\xi_{b,b'}(r)$, which we can think of as a ``local transition from $b$ to $b'$''. In this context the above conditions have the following meaning:
\begin{enumerate}
    \item $\gamma$-compatibility means that $\xi$-transitions occur in the context of $\gamma$ and $\beta$;
    \item Identity says that the transition from $b$ to $b$ is the identity;
    \item The cocycle condition simply ensures that transitions compose.
\end{enumerate}

\begin{example}[Local transition term in S5]\label{ex: local transition term}
We provide an example of a local transition term in \textsf{S5}. Let $\beta(q) \coloneqq \Diamond q\wedge\Diamond \neg q$, and $\gamma(q,r) \coloneqq \gamma_M(q,r) \vee \gamma_U(q,r)$, where
$$\gamma_M(q,r) \coloneqq \nabla (q \wedge r, \neg q \wedge \neg r) \quad \text{ and } \quad \gamma_U(q,r) \coloneqq \nabla (q \wedge \neg r, \neg q \wedge r).$$
Given propositional variables $q$ and $r$, define $\mu(q,r) \coloneqq (q \wedge r) \vee (\neg q \wedge \neg r)$; observe that $\neg \mu(q,r) = (q \wedge \neg r) \vee (\neg q \wedge r)$. We show that
$$\xi(q_1,q_2,r) \coloneqq \mu(q_2,\mu(q_1,r)) = (q_2 \wedge \mu(q_1,r)) \vee (\neg q_2 \wedge \neg \mu(q_1,r))$$
is a local transition term for $\gamma$.

Note that for a valuation $v$, we have that $v(\xi) = 1$ whenever ``$v(q_2) = 1$ if and only if $v(q_1) = v(r)$''; from this, (Identity) and (Cocycle) follow by unfolding the definitions. We show ($\gamma$-compatibility). We prove that $\gamma(q_1,r)\wedge\beta(q_2)\vdash\gamma(q_2,\xi)$. Let $\mathfrak{M},x$ be a model satisfying $\gamma(q_1,r)$ and $\beta(q_2)$; without loss of generality, assume that $\mathfrak{M},x$ satisfies $\gamma_M(q_1,r) = \nabla(q_1 \wedge r, \neg q_1 \wedge \neg r)$. To show that $\mathfrak{M},x$ satisfies $\gamma(q_2,\xi)$, we prove that $\mathfrak{M},x \Vdash \gamma_M(q_2,\xi)$. For what we said before,
$$\neg \xi(q_1,q_2,r) = \neg \mu(q_2,\mu(q_1,r)) = (q_2 \wedge \neg \mu(q_1,r)) \vee (\neg q_2 \wedge \mu(q_1,r)),$$
hence $\gamma_M(q_2,\xi) = \nabla(q_2 \wedge \xi, \neg q_2 \wedge \neg \xi) = \nabla(q_2 \wedge \mu(q_1,r), \neg q_2 \wedge \mu(q_1,r))$. Given $x' \in \mathfrak{M}$, by $\mathfrak{M},x' \Vdash \gamma_M(q_1,r)$, we have that $\mathfrak{M},x'$ satisfies one between: $q_1 \wedge r$ and $q_2$; $q_1 \wedge r$ and $\neg q_2$; $\neg q_1 \wedge \neg r$ and $q_2$; $\neg q_1 \wedge \neg r$ and $\neg q_2$. As a consequence, $\mathfrak{M},x'$ satisfies one between $q_2 \wedge \mu(q_1,r)$ and $\neg q_2 \wedge \mu(q_1,r)$. Vice versa, using that $\mathfrak{M},x$ satisfies $\Diamond q_2 \wedge\Diamond \neg q_2$, we can find $x' \in \mathfrak{M}$ such that $\mathfrak{M},x'$ satisfies $q_2$; using that $\mathfrak{M},x$ satisfies $\gamma_M(q_1,r) = \nabla(q_1 \wedge r, \neg q_1 \wedge \neg r)$, we also get that $\mathfrak{M},x'$ satisfies $\mu(q_1,r)$; putting everything together, $\mathfrak{M},x' \Vdash q_2 \wedge \mu(q_1,r)$. A similar argument proves the existence of $x'' \in \mathfrak{M}$ such that $\mathfrak{M},x'' \Vdash \neg q_2 \wedge \mu(q_1,r)$.
\end{example}

\begin{definition}
    Let $\alpha(\overline{p})$, $\beta(\overline{p},\overline{q})\in \mathsf{Ext}(\alpha)$ and $\gamma(\overline{p},\overline{q},\overline{r})\in \mathsf{Ext}(\beta)$, equipped with a sequence of local transition terms $\overline{\xi}$. The \emph{local coequivalence relation} $\rho_{\gamma,\overline{\xi}}$ associated with $\gamma$ and $\overline{\xi}$, is the coequivalence relation over $\gamma$,
    \begin{equation*}
        \rho_{\gamma,\overline{\xi}}=\gamma(\overline{p}_{1},\overline{q}_{1},\overline{r}_{1})\wedge \gamma(\overline{p}_{2},\overline{q}_{2},\overline{r}_{2})\wedge \bigwedge_{p \in \overline{p}} (p_{1}\leftrightarrow p_{2})\wedge \bigwedge_{r \in \overline{r}}(\xi_{r}(\overline{p}_{1},\overline{q}_{1},\overline{q}_{2},\overline{r}_{1})\leftrightarrow r_{2}).
    \end{equation*}
\end{definition}

Note that a local coequivalence relation $\rho_{\gamma,\overline{\xi}}$ is indeed a $\gamma$-coequivalence relation: it is in the domain $\gamma$ by $\gamma$-compatibility, satisfies reflexivity by (Identity), and is symmetric and transitive by (Cocycle).

\begin{remark}
As noted in Remark \ref{rem: sheaf-like condition}, the CoSP corresponds seemingly to a sheaf-like condition, ensuring the uniqueness of gluings of models over their restrictions. Local transition terms allow us to weaken this condition: rather than requiring the uniqueness of gluings for models over an even splitting of the propositional variables, we require it only for models which lie in the same path of the local transition, where
a model $\mathfrak{M}_1,x_1$ over $\overline{p}_1,\overline{q}_1,\overline{r}_1$ has a transition to $\mathfrak{M}_2,x_2$ over $\overline{p}_2,\overline{q}_2,\overline{r}_2$ if there exists $\mathfrak{N},y$ over the variables $\overline{p},\overline{q}_1,\overline{q}_2,\overline{r}$, such that $\mathfrak{N}^{\overline{p}_1=\overline{p},\overline{q}_1,\overline{r}_1=\overline{r}},y \cong \mathfrak{M}_1,x_1$ and $\mathfrak{N}^{\overline{p}_2=\overline{p},\overline{q}_2,\overline{r}_2 = \overline{\xi}(\overline{p},\overline{q}_1,\overline{q}_2,\overline{r})},y \cong \mathfrak{M}_2,x_2$.
\end{remark}

From this we extract the following property:

\begin{definition}
    A logic $L$ has the \emph{local coequivalence separation property} (LCoSP) if for each formula $\alpha(\overline{p})$, each $\beta(\overline{p},\overline{q})\in \mathsf{Ext}(\alpha)$, and each $\gamma(\overline{p},\overline{q},\overline{r})\in \mathsf{Ext}(\beta)$ together with a sequence of local transition terms $\overline{\xi}$ for $\gamma$, the local coequivalence relation $\rho_{\gamma,\overline{\xi}}$ separates variables. 
\end{definition}

\subsection{Descent data and Barr-exactness}

Local coequivalence relations correspond to a natural weakening of Barr-exactness \cite{JanelidzeSobralTholen2003}, using the notion of ``descent data", which we now introduce. We refer to \cite{JanelidzeSobralTholen2003,facetsI} for the following notions.

Let $\mathbf{C}$ be a category with pullbacks and $p : E \to B$ a morphism in $\mathbf{C}$. For an object $B$ in $\mathbf{C}$, let $\mathbf{C}/B$ denote the slice category over $B$, that is the category whose objects are pairs $(Y,g)$, with $g \colon Y \to B$ in $\mathbf{C}$, and whose morphisms $(Y,g) \to (Y',g')$ are morphisms $h \colon Y \to Y'$ in $\mathbf{C}$ such that $g' h = g$. There is a pair of adjoint functors
\[\begin{tikzcd}
	{p_! \dashv p^* : \mathbf{C}/B} & {\mathbf{C}/E}
	\arrow[from=1-1, to=1-2]
\end{tikzcd}\]
(with $p^*$ pulling back along $p$ and with $p_!$ composing with $p$ from the left). This induces a \emph{monad}\footnote{See e.g. \cite{monads} for basic concepts and definitions on monads, and algebras for a monad.} on $\mathbf{C}/E$. We denote by $\Des(p)$ the \textit{category of descent data for} $p$, i.e.\ the category of \emph{algebras}, in the sense of Eilenberg-Moore, for such a monad. An object of $\Des(p)$, called \textit{descent datum for} $p$, is then a triple $(X,f;\xi)$, with $(X,f) \in \mathbf{C}/E$ and $\xi \colon X \otimes_B E \to X$ in $\mathbf{C}$ satisfying the following compatibility conditions:
\begin{equation*}\label{eq:Des}
\begin{tikzcd}
	X & {X \otimes_B E} && {(X \otimes_B E) \otimes_B E} & {X \otimes_B E} \\
	X & E && {X \otimes_B E} & X
	\arrow["{{(1_X,f)}}", from=1-1, to=1-2]
	\arrow["{{1_X}}"', from=1-1, to=2-1]
	\arrow[""{name=0, anchor=center, inner sep=0}, "\xi"{description}, from=1-2, to=2-1]
	\arrow["{{\pi_2}}", from=1-2, to=2-2]
	\arrow["{{\xi \otimes_B 1_E}}", from=1-4, to=1-5]
	\arrow["{{\pi_1 \otimes_B 1_E}}"', from=1-4, to=2-4]
	\arrow["{(3)}"{description}, draw=none, from=1-4, to=2-5]
	\arrow["\xi", from=1-5, to=2-5]
	\arrow["f"', from=2-1, to=2-2]
	\arrow["\xi"', from=2-4, to=2-5]
	\arrow["{(2)}"{description}, draw=none, from=0, to=1-1]
	\arrow["{(1)}"{description}, draw=none, from=0, to=2-2]
\end{tikzcd}\tag{D}
\end{equation*}
A morphism $(X,f;\xi) \to (X',f';\xi')$ in $\Des(p)$ is a morphism $h \colon (X,f) \to (X',f')$ in $\mathbf{C}/E$ such that
\[\begin{tikzcd}
	{X \otimes_B E} & {X' \otimes_B E} \\
	X & {X'}
	\arrow["{h \otimes_B 1_E}", from=1-1, to=1-2]
	\arrow["\xi"', from=1-1, to=2-1]
	\arrow["{\xi'}", from=1-2, to=2-2]
	\arrow["h"', from=2-1, to=2-2]
\end{tikzcd}\]
commutes. This induces a comparison functor $\Phi^p \colon \mathbf{C}/B \to \Des(p)$, sending $(Y,g) \in \mathbf{C}/B$ to the descent data $(Y \otimes_B E, \pi_2; \xi)$, with $\xi$ given by $\pi_1 \otimes_B 1_E \colon (Y \otimes_B E) \otimes_B E \to Y \otimes_B E$.

\begin{definition}
We say that $p$ is a \emph{descent morphism} if $\Phi^p$ is full and faithful. If $\Phi^p$ is an equivalence of categories, then $p$ is said to be an \emph{effective descent morphism}.
\end{definition}

\begin{remark}
If $p$ is an (effective) descent morphism, then it is a universal regular epimorphism, meaning that its pullback along any morphism is a regular epimorphism. If $\mathbf{C}$ is Barr-exact, then the converse holds (see \cite{facetsI}). However, even assuming regularity, the fact that all regular epimorphisms are effective descent morphisms is not sufficient to ensure that $\mathbf{C}$ is Barr-exact \cite[pp.395,(ii)]{JanelidzeSobralTholen2003}.
\end{remark}

\begin{definition}
A regular category $\mathbf{C}$ is said to be \emph{almost Barr-exact} if all regular epimorphisms are effective descent morphisms.
\end{definition}

Almost Barr-exactness of duals of categories of algebras has been studied in the case of loops, magmas and groups \cite{Zangurashvili2003}. In regular categories, we have a nice characterization of effective descent morphisms:

\begin{theorem}\label{thm: all coequivalence relations are effective}
Let $\mathbf{C}$ be a regular category, $p\colon E\to B$ a morphism. Then:
\begin{enumerate}
    \item Given $(X,f;\xi) \in \Des(p)$ a descent datum, the morphism $\tilde{\xi} \coloneqq (\pi_1,\xi) \colon X \otimes_B E \to X \otimes X$ is an equivalence relation (the \emph{equivalence relation associated} to the descent datum $(X,f;\xi)$);
    \item If $p$ is a regular epimorphism, then it is an effective descent morphism if and only if for each $(X,f;\xi)\in \Des(p)$, the equivalence relation $\tilde{\xi}$ is effective.
\end{enumerate}
\end{theorem}
\begin{proof}
See \cite[Theorem 3.7 (c)]{JanelidzeSobralTholen2003}
\end{proof}

We say that an equivalence relation $R\hookrightarrow X\otimes X$ \emph{comes from a descent datum over} $f : X \to E$ \emph{for} $p : E \to B$ if it is isomorphic to $\tilde{\xi}$, for some $(X,f;\xi) \in \Des(p)$. The following appears to be folklore.

\begin{theorem}\label{thm:deseq}
Let $(\pi_1,\pi_2) : R\hookrightarrow X\otimes X$ be an equivalence relation. Then $R$ comes from a descent datum over $f : X \to E$ for $p : E \to B$ if, and only if,
\begin{equation*}\label{eq:deseq}
\begin{tikzcd}
	{R} & {X} \\
	{E} & {B}
	\arrow["{\pi_1}", from=1-1, to=1-2]
	\arrow["{f \pi_2}"', from=1-1, to=2-1]
	\arrow["{pf}", from=1-2, to=2-2]
	\arrow["p"', from=2-1, to=2-2]
\end{tikzcd}\tag{P}
\end{equation*}
is a diagram of pullback. If that is the case, $\pi_2 \colon R \to X$ is the descent datum whose associated equivalence relation is $R$.
\end{theorem}
\begin{proof}
This follows immediately from the equivalence between the category $\Des(p)$ introduced above and the category of discrete fibrations over the kernel pair of $p$ (see \cite[Section 5.A]{JanelidzeSobralTholen2003}).
\end{proof}

We thus refer to $R$ as a \textit{descent equivalence relation} if $R$ satisfies the above conditions relative to some $f$ and $p$.
\begin{corollary}\label{cor:almostchar}
A regular category $\mathbf{C}$ is almost Barr-exact if, and only if, all descent equivalence relations are effective.
\end{corollary}

\subsection{Local coequivalences and descent}
In this section we establish the connection between the logical property of LCoSP and descent. Recall from \ref{not: conventions} that we identify a formula $\alpha(\overline{p})$ with the object in $\mathsf{Alg}(L)_{fp}^{op}$ dual to the finitely presented algebra $\mathcal{F}_L(\overline{p})/\alpha(\overline{p})$ and that derivations $\beta(\overline{q}) \vdash_L \alpha(\overline{\theta}/\overline{p})$ represent morphisms from $\beta$ to $\alpha$ in $\mathsf{Alg}(L)_{fp}^{op}$.

To understand the connection, first we lay out all necessary elements, and relate them dually to their meaning in the previous section:

\begin{enumerate}
    \item We consider $\alpha(\overline{p})$ and $\beta(\overline{p},\overline{q})$ and a morphism $\beta(\overline{p},\overline{q}) \vdash_L \alpha(\overline{p})$ in $\mathsf{Alg}(L)_{fp}^{op}$ (read $p \colon E \to B$).
    \item Morphisms to $\beta$ correspond to $\gamma(\overline{p},\overline{q},\overline{r}) \vdash_L \beta(\overline{p},\overline{q})$ (read $f \colon X \to E$).
    \item For any such morphism, one can compute the pullback of $\gamma(\overline{p},\overline{q},\overline{r})\vdash_{L}\alpha(\overline{p})$ along $\beta(\overline{p},\overline{q})\vdash_{L}\alpha(\overline{p})$ (read $X\otimes_{B}E$). This comes equipped with two projections:
    \begin{equation*}
    \gamma(\overline{p},\overline{q}_1,\overline{r}) \wedge \beta(\overline{p},\overline{q}_2) \vdash_L \gamma(\overline{p},\overline{q}_1/\overline{q},\overline{r}) \text{ and } \gamma(\overline{p},\overline{q}_1,\overline{r}) \wedge \beta(\overline{p},\overline{q}_2) \vdash_L \beta(\overline{p},\overline{q}_2/\overline{q}).
    \end{equation*}
    \item A \emph{descent datum} over $\gamma(\overline{p},\overline{q},\overline{r}) \vdash_L \beta(\overline{p},\overline{q})$ for $\beta(\overline{p},\overline{q})\vdash_{L}\alpha(\overline{p})$ is therefore given by a \emph{connecting morphism}  $$\gamma(\overline{p},\overline{q}_{1},\overline{r})\wedge\beta(\overline{p},\overline{q}_{2})\vdash_{L}\gamma(\overline{\mu}/\overline{p},\overline{\nu}/\overline{q},\overline{\xi}/\overline{r}),$$ (read as $X\otimes_{B}E\to X$), making the diagrams (1-3) of (\ref{eq:Des}) commute.
\end{enumerate}

Note that in light of the previous point, we can see that \emph{local transition terms} have the precise shape necessary to be connecting morphisms. We now prove that they satisfy the necessary conditions:

\begin{proposition}
Let $L$ be normal modal logic,  $\alpha(\overline{p})$ a formula, $\beta(\overline{p},\overline{q})\in \mathsf{Ext}(\alpha)$, and $\gamma(\overline{p},\overline{q},\overline{r})\in \mathsf{Ext}(\beta)$. To give a morphism from $\gamma(\overline{p},\overline{q}_1,\overline{r}) \wedge \beta(\overline{p},\overline{q}_2)$ to $\gamma(\overline{p},\overline{q},\overline{r})$ in $\mathsf{Alg}(L)_{fp}^{op}$ that makes the diagram (1) of (\ref{eq:Des}) commute, is equivalent to giving a sequence $\overline{\xi} = (\xi_r(\overline{p},\overline{q}_1,\overline{q}_2,\overline{r}) \colon r \in \overline{r})$ that satisfies ($\gamma$-compatibility) of Definition \ref{def:transition}. Moreover:
\begin{enumerate}
    \item[a)] the diagram (2) of (\ref{eq:Des}) commutes if and only if $\overline{\xi}$ satisfies (Identity) of Definition \ref{def:transition};
    \item[b)] the diagram (3) of (\ref{eq:Des}) commutes if and only if $\overline{\xi}$ satisfies (Cocycle) of Definition \ref{def:transition}.
\end{enumerate}
\end{proposition}
\begin{proof}
As we said before, a morphism from $\gamma(\overline{p},\overline{q}_1,\overline{r}) \wedge \beta(\overline{p},\overline{q}_2)$ to $\gamma(\overline{p},\overline{q},\overline{r})$ in $\mathsf{Alg}(L)_{fp}^{op}$ is specified by sequences of formulas $\overline{\mu} = (\mu_p \colon p \in \overline{p})$, $\overline{\nu} = (\nu_q \colon q \in \overline{q})$ and $\overline{\xi} = (\xi_r \colon r \in \overline{r})$ in the variables $\overline{p},\overline{q}_1,\overline{q}_2,\overline{r}$. Commutativity of (1) establishes that the projection $\gamma(\overline{p},\overline{q}_1,\overline{r}) \wedge \beta(\overline{p},\overline{q}_2) \vdash_L \beta(\overline{p},\overline{q}_2/\overline{q})$ coincides with the composition of $\gamma(\overline{p},\overline{q}_1,\overline{r}) \wedge \beta(\overline{p},\overline{q}_2) \vdash_L \gamma(\overline{\mu}/\overline{p},\overline{\nu}/\overline{q},\overline{\xi}/\overline{r})$ with $\gamma(\overline{p},\overline{q},\overline{r}) \vdash_L \beta(\overline{p},\overline{q})$; the latter is $\gamma(\overline{p},\overline{q}_1,\overline{r}) \wedge \beta(\overline{p},\overline{q}_2) \vdash_L \beta(\overline{\mu}/\overline{p},\overline{\nu}/\overline{q})$ (compositions is simply the composition of substitutions). The two morphisms coincide if and only if, up to $\gamma(\overline{p},\overline{q}_1,\overline{r}) \wedge \beta(\overline{p},\overline{q}_2)$, the substitution $\overline{\mu}$ is equivalent to $\overline{p}$ and the substitution $\overline{\nu}$ is equivalent to $\overline{q}_2$; we can therefore simplify the specification of the initial morphism (as it presents the same morphism) in the following equivalent way
$$\gamma(\overline{p},\overline{q}_1,\overline{r}) \wedge \beta(\overline{p},\overline{q}_2) \vdash_L \gamma(\overline{p},\overline{q}_2/\overline{q},\overline{\xi}/\overline{r}),$$
which is exactly ($\gamma$-compatibility).

Moreover, the diagram (2) commutes if an only if the composition of $\gamma(\overline{p},\overline{q},\overline{r}) \vdash_L \gamma(\overline{p},\overline{q}/\overline{q}_1,\overline{r}) \wedge \beta(\overline{p},\overline{q}/\overline{q}_2)$ (read as $(1_x,f) \colon X \to X \otimes_B E$) and our morphism $\gamma(\overline{p},\overline{q}_1,\overline{r}) \wedge \beta(\overline{p},\overline{q}_2) \vdash_L \gamma(\overline{p},\overline{q}_2/\overline{q},\overline{\xi}/\overline{r})$, namely $\gamma(\overline{p},\overline{q},\overline{r}) \vdash_L \gamma(\overline{p},\overline{q},\overline{\xi}(\overline{p},\overline{q}/\overline{q}_1,\overline{q}/\overline{q}_2,\overline{r})/\overline{r})$, is the identity; this is equivalent to say that
$$\gamma(\overline{p},\overline{q},\overline{r}) \vdash_L r \leftrightarrow \xi_r(\overline{p},\overline{q}/\overline{q}_1,\overline{q}/\overline{q}_2,\overline{r})$$
for each $r \in \overline{r}$, which is exactly (Identity). The proof of b) is similar (observe that $(X \otimes_B E) \otimes_B E$ is $\gamma(\overline{p},\overline{q}_1,\overline{r}) \wedge \beta(\overline{p},\overline{q}_2) \wedge \beta(\overline{p},\overline{q}_3)$, with obvious projections).
\end{proof}

Recall that the local coequivalence relation associated to $\gamma$ and $\overline{\xi}$, is given explicitly by:
    \begin{equation*}
        \rho_{\gamma,\overline{\xi}}=\gamma(\overline{p}_{1},\overline{q}_{1},\overline{r}_{1})\wedge \gamma(\overline{p}_{2},\overline{q}_{2},\overline{r}_{2})\wedge \bigwedge_{p \in \overline{p}} (p_{1}\leftrightarrow p_{2})\wedge \bigwedge_{r \in \overline{r}}(\xi_{r}(\overline{p}_{1},\overline{q}_{1},\overline{q}_{2},\overline{r}_{1})\leftrightarrow r_{2}).
    \end{equation*}

\begin{proposition}\label{prop: correspondence of descent with local coequivalence}
The finitely presented algebra $\mathcal{F}_L(\overline{p}_{1},\overline{q}_{1},\overline{r}_{1},\overline{p}_{2},\overline{q}_{2},\overline{r}_{2})/\rho_{\gamma,\overline{\xi}}$ is dual to the equivalence relation associated to the descent datum $\gamma(\overline{p},\overline{q}_1,\overline{r}) \wedge \beta(\overline{p},\overline{q}_2) \vdash_L \gamma(\overline{p},\overline{q}_2/\overline{q},\overline{\xi}/\overline{r})$ over $\gamma(\overline{p},\overline{q},\overline{r}) \vdash_L \beta(\overline{p},\overline{q})$ for $\beta(\overline{p},\overline{q}) \vdash_L \alpha(\overline{p})$.
\end{proposition}
\begin{proof} Observe that the diagram (\ref{eq:deseq}) commutes. Indeed, the composition of $\rho_{\overline{\xi},\gamma} \vdash_L \gamma(\overline{p}_1/\overline{p},\overline{q}_1/\overline{q},\overline{r}_1/\overline{r})$ and $\gamma(\overline{p},\overline{q},\overline{r}) \vdash_L \alpha(\overline{p})$, which is $\rho_{\overline{\xi},\gamma} \vdash_L \alpha(\overline{p}_1/\overline{p})$, is equal to the composition of $\rho_{\overline{\xi},\gamma} \vdash_L \beta(\overline{p}_2/\overline{p},\overline{q}_2/\overline{q})$ and $\beta(\overline{p},\overline{q}) \vdash_L \alpha(\overline{p})$, which is $\rho_{\overline{\xi},\gamma} \vdash_L \alpha(\overline{p}_2/\overline{p})$: this is because $\rho_{\overline{\xi},\gamma} \vdash_L p_1 \leftrightarrow p_2$ for all $p \in \overline{p}$. By the universal property of the pullback, there exists a unique morphism from $\rho_{\overline{\xi},\gamma}$ to $\gamma(\overline{p},\overline{q}_1,\overline{r}) \wedge \beta(\overline{p},\overline{q}_2)$, making the projections commute; it is straightforward to verify that such a morphism is
\begin{equation*}\label{eq:comp}
\rho_{\gamma,\overline{\xi}}(\overline{p}_1,\overline{q}_1,\overline{r}_1,\overline{p}_2,\overline{q}_2,\overline{r}_2) \vdash_L \gamma(\overline{p}_1/\overline{p},\overline{q}_1,\overline{r}_1/\overline{r}) \wedge \beta(\overline{p}_1/\overline{p},\overline{q}_2).\tag{Univ}
\end{equation*}
(it is indeed a derivation, since $\rho_{\overline{\xi},\gamma} \vdash_L \bigwedge_{p \in \overline{p}} (p_1 \leftrightarrow p_2)$ and $\gamma(\overline{p}_2,\overline{q}_2,\overline{r}_2) \vdash_L \beta(\overline{p}_2,\overline{q}_2)$, using  substitution). By Theorem \ref{thm:deseq}, the claim follows once we prove that the morphism (\ref{eq:comp}) is an isomorphism and that, up to (\ref{eq:comp}), the second projection $\rho_{\overline{\xi},\gamma} \vdash_L \gamma(\overline{p}_2/\overline{p},\overline{q}_2/\overline{q},\overline{r}_2/\overline{r})$ is the descent datum $\gamma(\overline{p},\overline{q}_1,\overline{r}) \wedge \beta(\overline{p},\overline{q}_2) \vdash_L \gamma(\overline{p},\overline{q}_2/\overline{q},\overline{\xi}/\overline{r})$. To prove that (\ref{eq:comp}) is an isomorphism, it suffices to observe that
\begin{align*}\label{eq:compinv}
\tag{Univ$^{*}$}
\gamma(\overline{p},\overline{q}_1,\overline{r}) \wedge \beta(\overline{p},\overline{q}_2) &\vdash_L \rho_{\gamma,\overline{\xi}}(\overline{p}/\overline{p}_1,\overline{q}_1,\overline{r}/\overline{r}_1,\overline{p}/\overline{p}_2,\overline{q}_2,\overline{\xi}/\overline{r}_2)
\end{align*}
is its inverse: the above is a derivation by ($\gamma$-compatibility); the composite $(\mathrm{Univ}\circ \mathrm{Univ}^{*})$ sends each propositional letter to itself, and is therefore trivial; whilst the only non-trivial substitutions of the composition $(\mathrm{Univ}^{*}\circ \mathrm{Univ})$ are the following: $\overline{p}_{2}$ is sent to $\overline{p}_{1}$, which $\rho$ proves equivalent to $\overline{p}_{2}$, and $\overline{r}_{2}$ to $\xi(\overline{p}_{1},\overline{q}_{1},\overline{q}_{2},\overline{r}_{1})$ which $\rho$ proves equivalent to $\overline{r}_{2}$.

Now observe that the composition of (\ref{eq:compinv}) with the second projection $\rho_{\overline{\xi},\gamma} \vdash_L \gamma(\overline{p}_2/\overline{p},\overline{q}_2/\overline{q},\overline{r}_2/\overline{r})$ is exactly the descent datum $\gamma(\overline{p},\overline{q}_1,\overline{r}) \wedge \beta(\overline{p},\overline{q}_2) \vdash_L \gamma(\overline{p},\overline{q}_2/\overline{q},\overline{\xi}/\overline{r})$.
\end{proof}

\begin{theorem}\label{thm: characterization of almost Barr-exactness}
    For $L$ a logic with the Beth property, the following are equivalent:
    \begin{enumerate}
        \item $L$ has the local coequivalence separation property.
        \item In $\mathsf{Alg}(L)_{fp}^{op}$ all descent equivalence relations are effective.
    \end{enumerate}
    If $L$ has $\exists_{\mathrm{MIP}}$ and the Beth property, then this is equivalent to $\mathsf{Alg}(L)_{fp}^{op}$ being almost Barr-exact.
\end{theorem}
\begin{proof}
The fact that local coequivalences being separating coincides with descent equivalence relations being effective follows from Proposition \ref{prop: correspondence of descent with local coequivalence}. By Proposition \ref{prop: regularity of the category in question}, $\exists_{\mathrm{MIP}}$ corresponds to regularity of the category, and hence in such a setting, the category will be almost Barr-exact. 
\end{proof}

\section{Coequivalence separation in S5}\label{sec: descent equivalence relations in S5 frames}

In this final section we will focus on the case of the modal logic system $\mathsf{S5}$. We begin by outlining a representation of limits of finite $\mathsf{S5}$ frames, originally due to \cite{Bass1958} (see also \cite{Horn1978FreeSA}) which will be convenient for our purposes. We will establish a criterion for equivalence relations in $\mathsf{S5}$ to be effective. From this we will obtain an easy proof that $\mathsf{S5}$ fails to have CoSP. We will moreover show that LCoSP does hold.

\subsection{Representation}

Recall that a map $f \colon (X,R)\to (Y,S)$ between Kripke frames is called a \textit{p-morphism} if (1) whenever $xRy$ then $f(x)Sf(y)$ and (2) whenever $f(x)Sy$ then there is some $x'\in X$ such that $xRx'$ and $f(x')=y$.

\begin{definition}
    We denote by $\mathsf{S5Fr}$ the category of $\mathsf{S5}$-frames with p-morphisms, by $\mathsf{S5Fr}_{l.f}$ the full subcategory of locally finite $\mathsf{S5}$-frames, and by $\mathsf{S5Fr}_{f}$ the full subcategory of $\mathsf{S5}$ finite frames.
\end{definition}

It is obvious to see that a map $f\colon (X,R)\to (Y,S)$ between  simple (finite) $\mathsf{S5}$-frames is a p-morphism if and only if it is surjective. Thus the category of finite and simple $\mathsf{S5}$-frames is definitionally equivalent to the category $\mathbf{FinSet}_{\neq \emptyset}^s$
of non-empty finite sets and surjective functions between them. By dualising the universal algebraic idea of representing an algebra by its subdirectly irreducible algebras (in this case, all simple algebras), we can identify a locally finite $\mathsf{S5}$ frame with a \emph{family} $\mathcal{F}$ of non-empty finite sets. Under such an identification we have the following easy facts:
\begin{proposition}\label{prop: basic facts of the representation} Let $\mathcal{F}$ and $\mathcal{G}$ belong to $\mathsf{S5Fr}_{l.f}$ (or $\mathsf{S5Fr}_{f}$):
\begin{enumerate}
    \item\label{eq: definition of p-morphism} A p-morphism $f\colon \mathcal{F}\to \mathcal{G}$ is equivalent to a \emph{function} $f \colon \mathcal{F} \to \mathcal{G}$ and, a \emph{family of surjective functions}, indexed on $\mathcal{F}$, $\{f_X \colon X \to f(X)\ \vert\ X \in \mathcal{F}\}$.
    \item\label{eq: structure of subobjects} The subobjects of $\mathcal{G}$ are precisely the images of injective functions $f \colon \mathcal{F} \to \mathcal{G}$, with all $f_X$ bijective.
\end{enumerate}
\end{proposition}

For simplicity, the morphism will be denoted just by $f \colon \mathcal{F} \to \mathcal{G}$. The category $\Fr_{l.f}$ is then an instance of the general construction associating to a category $\mathsf{C}$ the \emph{category of families} $\Fam(\mathsf{C})$, with $\mathsf{C} = \mathbf{FinSet}_{\neq \emptyset}^s$; moreover, $\Fr_f$ is equivalent to the full subcategory $\Fam_f(\mathsf{C})$ of finite families (see \cite{essalg}). From now on, we will use the notation of families to refer to the objects and the morphisms of $\Fr_{l.f}$ and $\Fr_f$.


\begin{definition}
Let $X$ and $Y$ be non-empty finite sets. A \textit{bisimulation} from $X$ to $Y$ is a relation $A \subseteq X \times Y$ such that for each $x \in X$, there exists $y \in Y$ such that $(x,y) \in A$; and for each $y \in Y$, there exists $x \in X$ such that $(x,y) \in A$.
\end{definition}

In other words, the restrictions of the two projections $X\times Y\to X$ and $X\times Y\to Y$ to $A$ are both surjective functions. We denote by $\mathsf{B}(X,Y)$ the set of bisimulations from $X$ to $Y$. Bisimulations provide a nice way to describe pullbacks in $\Fr_{l.f}$ and $\Fr_f$.\footnote{This description goes back to \cite{Bass1958}, and \cite{Horn1978FreeSA}.} Consider the following diagram
\[\begin{tikzcd}
	& {\mathcal{W}} \\
	{\mathcal{F}} & {\mathcal{G}}
	\arrow["q", from=1-2, to=2-2]
	\arrow["p"', from=2-1, to=2-2]
\end{tikzcd}\]
in either $\Fr_{l.f}$ or $\Fr_f$. 

\begin{definition}
We say that $X \in \mathcal{W}$ and $U \in \mathcal{F}$ are $(q,p)$-\emph{compatible} if $q(X) = p(U)$. In such a case we say that $D \in \mathsf{B}(X,U)$ is $(q,p)$-\emph{admissible} if, for all $(x,u) \in D$, we have $q_X(x) = p_U(u)$. Define $\mathsf{Bis}(q,p)$ as the disjoint union
$$\mathsf{Bis}(p,q)\coloneqq \bigcup\{D \in \mathsf{B}(X,U)\ \vert\ D \text{ is } (q,p)\text{-admissible and } X,U \text{ are $(q,p)$-compatible}\}.$$
\end{definition}

From now on, $\mathsf{Bis}(q,p)$ will be seen as an object of $\Fr_{l.f}$ (each $D \in \mathsf{Bis}(q,p)$ is a non-empty finite set, being a bisimulation from some $X \in \mathcal{W}$ to some $U \in \mathcal{F}$) and of $\Fr_f$, in case both $\mathcal{W}$ and $\mathcal{F}$ are finite. Consider $D \in \mathsf{Bis}(q,p)$, say $D \in \mathsf{B}(X,U)$ $(q,p)$-admissible with $X \in \mathcal{W}$ and $U \in \mathcal{F}$ $(q,p)$-compatible. Set $\pi_1(D) := X$ and $\pi_2(D) := U$ and let
$$(\pi_1)_{D} : D \to \pi_1(D) = X \quad \text{ and } \quad (\pi_2)_{D} : D \to \pi_2(D) = U$$
be the corresponding projections. Since $(\pi_1)_{D}$ and $(\pi_2)_{D}$ are surjective functions, $D$ being a bisimulation, the data above define a pair of p-morphisms $\pi_1 : \mathsf{Bis}(q,p) \to \mathcal{W}$ and $\pi_2 : \mathsf{Bis}(q,p) \to \mathcal{F}$ (by Proposition \ref{prop: basic facts of the representation}.\ref{eq: definition of p-morphism}.).
\begin{proposition}\label{prop:pullback given by category representation}
The diagram
\[\begin{tikzcd}
	{\mathsf{Bis}(q,p)} & {\mathcal{W}} \\
	{\mathcal{F}} & {\mathcal{G}}
	\arrow["{\pi_1}", from=1-1, to=1-2]
	\arrow["{\pi_2}"', from=1-1, to=2-1]
	\arrow["q", from=1-2, to=2-2]
	\arrow["p"', from=2-1, to=2-2]
\end{tikzcd}\]
is a pullback in $\Fr_{l.f}$ (and in $\Fr_f$, if $\mathcal{W}$, $\mathcal{F}$ and $\mathcal{G}$ are finite).
\end{proposition}
\begin{proof}
Consider p-morphisms $f_1 \colon \mathcal{V} \to \mathcal{W}$ and $f_2 \colon \mathcal{V} \to \mathcal{F}$ such that $qf_1 = pf_2$. For each $Y \in \mathcal{V}$, we have $f_1(Y) \in \mathcal{W}$ and $f_2(Y) \in \mathcal{F}$, and the pair of surjective functions $(f_i)_Y \colon Y \to f_i(Y)$ induce a unique function $Y \to f_1(Y) \times f_2(Y)$; call $f(Y) \subseteq f_1(Y) \times f_2(Y)$ its image. Note that because $(f_1)_Y$ and $(f_2)_Y$ are surjective, we have that $f(Y)$ is a bisimulation from $f_1(Y)$ to $f_2(Y)$. Moreover, by $qf_1 = pf_2$, we have that $f_1(Y)$ and $f_2(Y)$ are $(q,p)$-compatible and that $f(Y)$ is $(q,p)$-admissible; in other words, $f(Y) \in \mathsf{Bis}(q,p)$. Let $f_Y \colon Y \rightarrow f(Y)$ be the restriction of the $Y \rightarrow f_1(Y) \times f_2(Y)$ above on the codomain. By Proposition \ref{prop: basic facts of the representation}.\ref{eq: definition of p-morphism}., the data above define a p-morphism $f \colon \mathcal{V} \rightarrow \mathsf{Bis}(q,p)$, whose uniqueness is easily verified.
\end{proof}

\begin{notation}
    When $\mathcal{G}$ is the terminal object, i.e.\ the family $\mathbf{1} := \{\{*\}\}$, then $\mathsf{Bis}(q,p)$ is just the disjoint union, varying $X \in \mathcal{W}$ and $U \in \mathcal{F}$, of the sets $\mathsf{B}(X,U)$ ($(q,p)$-compatibility and $(q,p)$-admissibility trivialize) it represents the product of $\mathcal{W}$ and $\mathcal{F}$. We then denote $\mathsf{Bis}(q,p)$ by   $\mathsf{Bis}(\mathcal{W},\mathcal{F})$. Moreover, if $\mathcal{W} = \mathcal{F}$, we simply write $\mathsf{Bis}(\mathcal{W})$.
\end{notation}

We will need a number of special operations on bisimulations.

\begin{definition}
    Let $\mathcal{W}$ be a family, $X,Y,Z\in \mathcal{W}$, $A\in \mathsf{B}(X,Y)$ and $B\in \mathsf{B}(Y,Z)$. We say that $P\in \mathsf{B}(A,B)$ is a \emph{bisimulation of composable pairs} if for each $((x,y),(y',z))\in P$ we have that $y=y'$. If $P$ is a bisimulation of composable pairs, we denote by $\mathsf{Comp}(P)=\{(x,z)\in X\times Z :\exists y\in Y, ((x,y),(y,z))\in P\}$.
\end{definition}

\begin{proposition}[Closure properties of bisimulations]\label{prop: closure properties of bisimulations}
Given $\mathcal{W}\in \mathsf{S5Fr}_{l.f}$ (or $\mathsf{S5Fr}_{f}$), and $X,Y,Z\in \mathcal{W}$, we have:
\begin{enumerate}
    \item For each $X\in \mathcal{W}$, the relation $\Delta_{X}=\{(x,x) \in X\times X : x\in X\}$ is in $\mathsf{B}(X,X)$.
    \item For each $A\in \mathsf{B}(X,Y)$, we have $A^{\dagger} := \{(y,x) \in Y \times X\ \vert\ (x,y) \in A\}$ and $A^{\dagger}\in \mathsf{B}(Y,X)$.
    \item If $A\in \mathsf{B}(X,Y)$ and $B\in \mathsf{B}(Y,Z)$, the relation $A B := \{(x,z) \in X \times Z\ \vert\ \exists y \in Y ((x,y) \in A \text{ and } (y,z) \in B)\}$ is in $\mathsf{B}(X,Z)$;
    \item \label{eq: composable} If $A\in \mathsf{B}(X,Y)$ and $B\in \mathsf{B}(Y,Z)$, and $P\in \mathsf{B}(A,B)$, where $P$ is a bisimulation of composable pairs, then $\mathsf{Comp}(P)\in \mathsf{B}(X,Z)$.
\end{enumerate}
\end{proposition}
\begin{proof}
    We show \ref{eq: composable}., the other points having similarly easy proofs. If $x \in X$, then because $A$ is a bisimulation, there is some $y \in Y$ such that $(x,y)\in A$; since $P$ is a bisimulation, there is then some $(y',z) \in B$ such that $((x,y),(y',z))\in P$, which forces $y = y'$ (because $P$ is a bisimulation of composable pairs). We found $z \in Z$ such that $(x,z)\in \mathsf{Comp}(P)$. The other condition is similarly verified.
\end{proof}

\begin{remark}
    Note that the composition law above is stronger than simple composition of bisimulations. Given $X,Y,Z\in \mathcal{W}$, $A\in \mathsf{B}(X,Y)$ and $B\in \mathsf{B}(Y,Z)$ and $A,B\in \mathcal{E}$, then $AB=\{(x,z) : \exists y, (x,y)\in A\text{ and } (y,z)\in B\}$ is $\mathsf{Comp}(A*B)$ where $A*B=\{((x,y),(y,z)) : (x,y)\in A, (y,z)\in B\}$, which is a bisimulation of composable pairs, since $A,B$ are bisimulations. We note that the weaker condition, which we refer to as \emph{(W-Composition)}, that for $A,B\in \mathcal{E}$ we have $AB\in \mathcal{E}$ will often be sufficient in the proofs below. 
\end{remark}

We also note the following basic fact (see e.g. \cite{Tarski1941}):

\begin{proposition}
    For each $\mathcal{W} \in \mathsf{S5Fr}_{l.f}$, and $X,Y\in \mathcal{W}$ we have that $\mathsf{B}(X,Y)$ forms a monoid: given $A, A' \in \mathsf{B}(X,Y)$, if $A\subseteq A'$, then $AB\subseteq A'B$ and $CA\subseteq CA'$, whenever the compositions are defined. Moreover, we have $\Delta_X \subseteq A A^{\dagger}$ and $\Delta_Y \subseteq A^{\dagger} A$.
\end{proposition}

Recall from Proposition \ref{prop: basic facts of the representation}.\ref{eq: structure of subobjects}. that the lattice $\mathsf{Sub}(\mathcal{G})$ of subobjects of $\mathcal{G}$ is (isomorphic to) the lattice of subsets $\mathcal{F} \subseteq \mathcal{G}$. Equivalence relations over $\mathcal{W}$ (see Definition \ref{def:eqrel}) are thus special objects of $\mathsf{Sub}(\mathsf{Bis}(\mathcal{W}))$.

\begin{definition}\label{def:equivalential}
    We say that a subobject $\mathcal{E}\in \mathsf{Sub}(\mathsf{Bis}(\mathcal{W}))$ is \emph{equivalential} if:
    \begin{enumerate}
    \item (Identity) For each $X\in \mathcal{W}$, $\Delta_{X}\in \mathcal{E}$;
        \item (Dagger) Whenever $X,Y\in \mathcal{W}$ and $A\in \mathsf{B}(X,Y)$ and $A\in \mathcal{E}$, then $A^{\dagger}\in \mathcal{E}$;
        \item (Composition) Whenever $X,Y,Z\in \mathcal{W}$, $A\in \mathsf{B}(X,Y)$ and $B\in \mathsf{B}(Y,Z)$, and $A,B\in \mathcal{E}$, and $P\in \mathsf{B}(A,B)$ is a bisimulation of composable pairs, then $\mathsf{Comp}(P)\in \mathcal{E}$.
    \end{enumerate}
\end{definition}

\begin{theorem}\label{thm: characterization of equivalence relations}
    Let $\mathcal{W}\in \mathsf{S5Fr}_{l.f}$. The following are equivalent for $\mathcal{E}\in \mathsf{Sub}(\mathsf{Bis}(\mathcal{W}))$:
    \begin{enumerate}
        \item $\mathcal{E}$ is an equivalential subobject of $\mathsf{Bis}(\mathcal{W})$;
        \item $\mathcal{E}$ is an equivalence relation on $\mathcal{W}$.
    \end{enumerate}
\end{theorem}
\begin{proof}
    Assume that $\mathcal{E}$ is a subobject of $\mathsf{Bis}(\mathcal{W})$. We show that:
    \begin{enumerate}
        \item[a)] $\mathcal{E}$ satisfies (Identity) if and only if it satisfies internal reflexivity;
        \item[b)] $\mathcal{E}$ satisfies (Dagger) if and only if it satisfies internal symmetry;
        \item[c)] $\mathcal{E}$ satisfies (Composition) if and only if it satisfies internal transitivity.
    \end{enumerate}

    We prove c), since a) and b) can be verified in a similar, albeit easier, manner. Recall that the categorical product $\mathcal{W} \otimes \mathcal{W}$ (represented by $\mathsf{Bis}(\mathcal{W})$, with projections $\pi_i \colon \mathsf{Bis}(\mathcal{W}) \to \mathcal{W}$) is an equivalence relation over $\mathcal{W}$ (it is the maximal one, meaning that all the others are subobjects of $\mathcal{W} \otimes \mathcal{W}$). Consider what we denoted by $(\mathcal{W} \otimes \mathcal{W}) \otimes_{\mathcal{W}} (\mathcal{W} \otimes \mathcal{W})$, namely the pullback of $\pi_2 \colon \mathsf{Bis}(\mathcal{W}) \to \mathcal{W}$ along $\pi_1 \colon \mathsf{Bis}(\mathcal{W}) \to \mathcal{W}$; it is represented by $\mathsf{Bis}(\pi_2,\pi_1)$, with projections $\kappa_i \colon \mathsf{Bis}(\pi_2,\pi_1) \to \mathsf{Bis}(\mathcal{W})$. Unraveling the definitions of $(\pi_2,\pi_1)$-compatibility and $(\pi_2,\pi_1)$-admissibility, we obtain that $\mathsf{Bis}(\pi_2,\pi_1)$ is the disjoint union, varying $X,Y,Z \in \mathcal{W}$, $A \in \mathsf{B}(X,Y)$ and $B \in \mathsf{B}(Y,Z)$, of the sets of bisimulations $P \in \mathsf{B}(A,B)$ of composable pairs; then, the projection $\kappa_1$ sends $P$ to $A$ and $(\kappa_1)_P \colon P \to A$ sends $((x,y),(y,z))$ to $(x,y)$ ($\kappa_2$ acts similarly). By Proposition \ref{prop: basic facts of the representation}.\ref{eq: definition of p-morphism}., the family of surjective functions $P \to \mathsf{Comp}(P)$, sending $((x,y),(y,z))$ to $(x,z)$, defines a p-morphism $\mathsf{Comp} \colon \mathsf{Bis}(\pi_2,\pi_1) \to \mathsf{Bis}(\mathcal{W})$. It is straightforward to show that $\pi_1 \mathsf{Comp} = \pi_1 \kappa_1$ and $\pi_2 \mathsf{Comp} = \pi_2 \kappa_2$, proving that $\mathsf{Comp}$ is the morphism witnessing internal transitivity for the maximal relation $\mathcal{W} \otimes \mathcal{W}$ (as required by Definition \ref{def:eqrel} of categorical equivalence relation). For any other subobject $\mathcal{E} \in \mathsf{Sub}(\mathcal{W} \otimes \mathcal{W})$, the pullback $\mathcal{E} \otimes_{\mathcal{W}} \mathcal{E}$ is a subobject of $\mathsf{Bis}(\pi_2,\pi_1)$; it is given by those bisimulations $P \in \mathsf{B}(A,B)$ of composable pairs for which $A, B \in \mathcal{E}$. Then, $\mathcal{E}$ satisfies internal transitivity if and only if the restriction of $\mathsf{Comp}$ to $\mathcal{E} \otimes_{\mathcal{W}} \mathcal{E}$ factorizes through $\mathcal{E}$ (see Remark \ref{rem:eqrel}); but this is exactly (Composition) of Definition \ref{def:equivalential} of equivalential subobject.
\end{proof}

Let $X,Y$ be non-empty finite sets. Note that in addition to the closure properties from Proposition \ref{prop: closure properties of bisimulations}, $\mathsf{B}(X,Y)$ is closed under \textit{non-empty}\footnote{By our definition, the empty relation is not a bisimulation; and the empty union of bisimulations would need to be the empty relation.} unions; moreover the unions are compatible with the composition relation, meaning that $A(B \cup B') = AB \cup AB'$ and $(C \cup C')A = CA \cup C'A$, whenever the compositions make sense. This observation allows us to characterize effectiveness of equivalence relations (see Definition \ref{def:effeqrel}):

\begin{proposition}\label{prop:eff}
Let $\mathcal{E}$ be an equivalential subobject of $\mathsf{Bis}(\mathcal{W})$. Then $\mathcal{E}$, as equivalence relation, is effective if, and only if, it is
\begin{enumerate}
    \item (downward closed) If $X, Y \in \mathcal{W}$, $A, B \in \mathsf{B}(X,Y)$, $A \in \mathcal{E}$ and $B \subseteq A$, then $B \in \mathcal{E}$;
    \item (closed under non-empty unions) If $X, Y \in \mathcal{W}$, $A, B \in \mathsf{B}(X,Y)$ and $A, B \in \mathcal{E}$, then $A \cup B \in \mathcal{E}$.
\end{enumerate}
\end{proposition}
\begin{proof}
Consider a p-morphism $f : \mathcal{W} \to \mathcal{F}$. The pullback $\mathsf{Bis}(f,f)$ of $f$ along itself is a subobject of $\mathsf{Bis}(\mathcal{W})$; by definition, whenever $X, Y \in \mathcal{W}$ and $A \in \mathsf{B}(X,Y)$, we have
$$A \in \mathsf{Bis}(f,f) \iff f(X) = f(Y) \text{ and } (\forall (x,y) \in A) (f_X(x) = f_Y(y))$$
It is then straightforward to prove that $\mathsf{Bis}(f,f)$ is downward closed and closed under non-empty unions, given that belonging to this pullback corresponds to a condition on pairs $(x,y)$.

Vice versa, let $\mathcal{E}$ be an equivalential subobject of $\mathsf{Bis}(\mathcal{W})$ and assume that it is downward closed and closed under non-empty unions (these assumptions will be used only in the very last part of the proof). Given $X, Y \in \mathcal{W}$, consider the family $\mathsf{B}(X,Y) \cap \mathcal{E}$ of the bisimulations from $X$ to $Y$ that belong to $\mathcal{E}$. Call $X$ and $Y$ $\mathcal{E}$-\textit{connected} if $\mathsf{B}(X,Y) \cap \mathcal{E}$ is non-empty. Being $\mathcal{E}$-connected is an equivalence relation over the set $\mathcal{W}$ (by (Identity), (Dagger) and (W-Composition)); call $\mathcal{W}/\mathcal{E}$ the set of its equivalence classes. If $K \in \mathcal{W}/\mathcal{E}$ is one of such classes, then we can consider the disjoint union $\{(x,X)\ \vert\ X \in K, x \in X\}$ and say that $(x,X)$ is equivalent to $(y,Y)$ if there is $A \in \mathsf{B}(X,Y) \cap \mathcal{E}$ such that $(x,y) \in A$ (it is indeed an equivalence relation, by (Identity), (Dagger) and (W-Composition)); call $K_\mathcal{E}$ the resulting quotient.

Consider the family $\mathcal{W}_\mathcal{E} := \{K_\mathcal{E}\ \vert\ K \in \mathcal{W}/\mathcal{E}\}$ and $f : \mathcal{W} \to \mathcal{W}_\mathcal{E}$ sending $X \in \mathcal{W}$ to $K_\mathcal{E}$, if $K \in \mathcal{W}/\mathcal{E}$ and $X \in K$, with $f_X : X \to f(X) = K_\mathcal{E}$ sending $x \in X$ to the equivalence class of $(x,X)$. Each $f_X$, for $X \in \mathcal{W}$, is surjective: consider an element in $K_\mathcal{E}$, say the class of $(y,Y)$, with $Y \in K$ and $y \in Y$; since $X \in K$, then $X$ and $Y$ are $\mathcal{E}$-connected, hence we can find $A \in \mathsf{B}(X,Y) \cap \mathcal{E}$ and, $A$ being a bisimulation, $x \in X$ such that $(x,y) \in A$. Then $f_X(x)$ is the class of $(x,X)$, which is equal to the class of $(y,Y)$. The fact that all $f_X$ are surjective proves at the same time that $\mathcal{W}_\mathcal{E}$ belongs to $\mathsf{S5Fr}_{l.f}$ (it is a family of non-empty finite sets) and that $f$ defines a p-morphism.\footnote{It can be proven that $f : \mathcal{W} \to \mathcal{W}_\mathcal{E}$ is the coequalizer of $\mathcal{E} \rightrightarrows \mathcal{W}$.} We can now compare $\mathcal{E}$ with the kernel pair $\mathsf{Bis}(f,f)$. Consider $X, Y \in \mathcal{W}$ and $A \in \mathsf{B}(X,Y)$. If $A \in \mathcal{E}$, then $A \in \mathsf{B}(X,Y) \cap \mathcal{E}$, hence $f(X) = f(Y)$ ($X$ and $Y$ are $\mathcal{E}$-connected), and, given $(x,y) \in A$, we have $f_X(x) = f_Y(y)$. Vice versa, assume that $f(X) = f(Y)$ and that $f_X(x) = f_Y(y)$ for each $(x,y) \in A$. The former condition says that $X$ and $Y$ are $\mathcal{E}$-connected, i.e.\ that $\mathsf{B}(X,Y) \cap \mathcal{E}$ is non-empty; we can then consider its greatest element $M \in \mathsf{B}(X,Y) \cap \mathcal{E}$, by closure under non-empty unions (take the union of all the bisimulations in $\mathsf{B}(X,Y) \cap \mathcal{E}$). The latter says that, for each $(x,y) \in A$, there exists $B \in \mathsf{B}(X,Y) \cap \mathcal{E}$ such that $(x,y) \in B$; in other words, $A \subseteq M$. We conclude that $A \in \mathcal{E}$, by downward closure.
\end{proof}

\begin{theorem}\label{thm: failure of CoSP for S5}
The system $\mathsf{S5}$ lacks the \emph{CoSP}\footnote{It should be noted that Theorem \ref{thm: failure of CoSP for S5} likewise follows  from the results of \cite{barrcoexactness}: for locally tabular logics $L$, if $L^{\infty}$ fails to have the CoSP, then $L$ will likewise fail to have it, where $L^{\infty}$ is a natural infinitary analogue of $L$ (see \cite{deberardinisprooftheoryforprofinite}), whenever the counterexample is given by finite frames. As a consequence of the main result of that paper, $\mathsf{S5}^{\infty}$ fails to have that property, and the counterexample is finite, from which the result for \textsf{S5} follows.}.
\end{theorem}
\begin{proof}
In Example \ref{ex: s5 non separating} we provided a coequivalence relation which is not separating. In light of the characterization of effectiveness, we can this easily: the equivalence relation associated to the formula fails to be closed under non-empty unions.
\end{proof}

\subsection{Local coequivalences in S5}

We now turn our attention to equivalence relations arising from descent data. For that purpose, we begin by  rephrasing \ref{thm:deseq} in our setting. Let $f : \mathcal{W} \to \mathcal{F}$ be a morphism in $\Fr_{l.f}$. Then, given $X, Y \in \mathcal{W}$ and $A \in \mathsf{B}(X,Y)$, we have a bisimulation $A_f \in \mathsf{B}(X,f(Y))$ defined as
$$A_f := \{(x,f_Y(y))\ \vert\ (x,y) \in A\}.$$
We have an obvious surjective function $A \to A_f$. These data define a p-morphism $\mathsf{Bis}(\mathcal{W}) \to \mathsf{Bis}(\mathcal{W},\mathcal{F})$. Observe that we can \emph{push $(-)_f$ to the right}: if $A$ and $B$ are composable, then $(AB)_f = A B_f$.
\begin{proposition}\label{prop:deseqframes}
Let $\mathcal{E}$ be an equivalential subobject of $\mathsf{Bis}(\mathcal{W})$. Then $\mathcal{E}$ comes from a descent data over $f : \mathcal{W} \to \mathcal{F}$ for $p : \mathcal{F} \to \mathcal{G}$ if and only if:
\begin{enumerate}
    \item \label{eq: first condition for descent} for each $A \in \mathcal{E}$, the function $A \to A_f$ is bijective and $A_f \in \mathsf{Bis}(pf,p)$;
    \item \label{eq: second condition for descent} for each $D \in \mathsf{Bis}(pf,p)$, there exists a unique $A \in \mathcal{E}$ such that $A_f = D$.
\end{enumerate}
\end{proposition}
\begin{proof}
Recall that $\mathsf{Bis}(pf,p)$ is the pullback of $pf$ along $p$ and it is a subobject of $\mathsf{Bis}(\mathcal{W},\mathcal{F})$. The fact that, for all $A \in \mathcal{E}$, we have $A_f \in \mathsf{Bis}(pf,p)$ means that $(pf)\pi_1 = p(f\pi_2)$; then $A \to A_f$ is the map $\mathcal{E} \to \mathsf{Bis}(pf,p)$ induced by the universal property. The other conditions establish bijectivity of such a map. The fact that this is equivalent with $\mathcal{E}$ arising from descent data then follows from Theorem \ref{thm:deseq}.
\end{proof}

\begin{definition}[Lift of bisimulation]
    Given $\mathcal{E}$  an equivalential subobject of $\mathsf{Bis}(\mathcal{W})$ arising as descent data, and $D \in\mathsf{Bis}(pf,p)$, i.e.\ $D \in \mathsf{B}(X,U)$ is $(pf,p)$-admissible, for some $X \in \mathcal{W}$ and $U \in \mathcal{F}$ a $(pf,p)$-compatible pair. We denote by $\mathsf{Lift}(D)$ the \emph{lift of $D$} the bisimulation uniquely determined by $\mathcal{E}$ relative to $D$, as in Proposition \ref{prop:deseqframes}.\ref{eq: second condition for descent}.: this means that $\mathsf{Lift}(D) \in \mathcal{E}$ and $(\mathsf{Lift}(D))_f = D$.
\end{definition}

Given $X \in \mathcal{W}$ and $U \in \mathcal{F}$ $(pf,p)$-compatible, i.e.\ such that $pf(X) = p(U)$, there exists the greatest $(pf,p)$-admissible bisimulation in $\mathsf{B}(X,U)$, namely
$$D_{X,U} := \{(x,u) \in X \times U\ \vert\ p_{f(X)} f_X(x) = p_U(u)\}.$$
Observe that $\mathsf{Lift}(D_{X,U}) \in \mathsf{B}(X,Z)$, where $Z$ is some element of $\mathcal{W}$ such that $f(Z) = U$.

With these observations we are ready for the main result of this section and of this paper:

\begin{theorem}\label{thm:desimplieseff}
In $\Fr_{l.f}$ and $\Fr_{f}$ all descent equivalence relations are effective.
\end{theorem}
\begin{proof}
Let $\mathcal{E}$ be an equivalential subobject of $\mathsf{Bis}(\mathcal{W})$, for some $\mathcal{W}$ in $\Fr_{l.f}$, and assume that $\mathcal{E}$ comes from a descent data over some $f : \mathcal{W} \to \mathcal{F}$ for some $p : \mathcal{F} \to \mathcal{G}$. We prove that $\mathcal{E}$ is effective.

We use the characterization of effectiveness in Proposition \ref{prop:eff}. Let $X, Y \in \mathcal{W}$. We first prove that, if $\mathsf{B}(X,Y) \cap \mathcal{E}$ is non-empty, then it has a greatest element. Take $A \in \mathsf{B}(X,Y) \cap \mathcal{E}$. By Proposition \ref{prop:deseqframes}.\ref{eq: first condition for descent}., since $A \in \mathcal{E}$, we have that $A_f \in \mathsf{Bis}(pf,p)$; this means that $X$ and $f(Y)$ are $(pf,p)$-compatible and that $A_f \in \mathsf{B}(X,f(Y))$ is $(pf,p)$-admissible. By $(pf,p)$-compatiblity, we can consider $D_{X,f(Y)}$ and $D_{Y,f(Y)}$, together with $\mathsf{Lift}(D_{X,f(Y)})$ and $\mathsf{Lift}(D_{Y,f(Y)})$ determined by $\mathcal{E}$. We prove that $A^{\dagger} \mathsf{Lift}(D_{X,f(Y)}) = \mathsf{Lift}(D_{Y,f(Y)})$: by uniqueness of Proposition \ref{prop:deseqframes}.\ref{eq: second condition for descent}., since both $A^{\dagger} \mathsf{Lift}(D_{X,f(Y)})$ and $\mathsf{Lift}(D_{Y,f(Y)})$ belong to $\mathcal{E}$ (by (Dagger) and (W-Composition)), it suffices to prove that $(A^{\dagger} \mathsf{Lift}(D_{X,f(Y)}))_f = (\mathsf{Lift}(D_{Y,f(Y)}))_f$, which is equivalent to $A^{\dagger} D_{X,f(Y)} = D_{Y,f(Y)}$, by pushing $(-)_f$ to the right. Obviously, the left-to-right inclusion holds, since $A^{\dagger} D_{X,f(Y)} \in \mathsf{B}(Y,f(Y))$ is $(pf,p)$-admissible and $D_{Y,f(Y)}$ is the greatest $(pf,p)$-admissible bisimulation in $\mathsf{B}(Y,f(Y))$. For the other inclusion, take $(y,u) \in D_{Y,f(Y)}$, i.e.\ such that $p_{f(Y)} f_Y(y) = p_{f(Y)}(u)$. Using the fact that $A$ is a bisimulation from $X$ to $Y$, we can find $x \in X$ such that $(x,y) \in A$. Since $A_f$ is $(pf,p)$-admissible, we have that $p_{f(X)} f_X (x) = p_{f(Y)} f_Y (y)$. As a consequence, $(x,u) \in D_{X,f(Y)}$. Composing, we have that $(y,u) \in A^{\dagger} D_{X,f(Y)}$. Putting everything together, since $\Delta_X \subseteq \mathsf{Lift}(D_{X,f(Y)}) (\mathsf{Lift}(D_{X,f(Y)}))^{\dagger}$,
\begin{align*}
A &\subseteq \mathsf{Lift}(D_{X,f(Y)}) (\mathsf{Lift}(D_{X,f(Y)}))^{\dagger} A =\\
&= \mathsf{Lift}(D_{X,f(Y)}) (A^{\dagger} \mathsf{Lift}(D_{X,f(Y)}))^{\dagger} = \mathsf{Lift}(D_{X,f(Y)}) (\mathsf{Lift}(D_{Y,f(Y)}))^{\dagger}.
\end{align*}
We proved that, if $\mathsf{B}(X,Y) \cap \mathcal{E}$ is non-empty, then $\mathsf{Lift}(D_{X,f(Y)}) (\mathsf{Lift}(D_{Y,f(Y)}))^{\dagger}$ is its greatest element.

We now prove that $\mathcal{E}$ is downward closed. Let $X, Y \in \mathcal{W}$ and let $A, B \in \mathsf{B}(X,Y)$ such that $B \subseteq A$; assume that $A \in \mathcal{E}$. We want to prove that $B \in \mathcal{E}$. Consider $M$ greatest element of $\mathsf{B}(X,Y) \cap \mathcal{E}$ (it exists for what we said before, $\mathsf{B}(X,Y) \cap \mathcal{E}$ being non-empty). Observe that
\begin{align*}
M \subseteq M M^{\dagger} B \subseteq M M^{\dagger} M \subseteq M.\tag{$*$}
\end{align*}

The first inclusion follows from the fact that $\Delta_Y \subseteq M^{\dagger} B$, since $B \subseteq A \subseteq M$ ($A \in \mathsf{B}(X,Y) \cap \mathcal{E}$ and $M$ is its greatest element); by $B \subseteq M$, we also have the second inclusion (the inclusion order is compatible with the composition); the last inclusion follows from the fact that $M M^{\dagger} M \in \mathsf{B}(X,Y) \cap \mathcal{E}$ (by (Dagger) and (W-Composition)) and that $M$ is its greatest element. Observe that, by the $(pf,p)$-admissibility of $A_f$, we also have that $B_f \in \mathsf{B}(X,f(Y))$ is $(pf,p)$-admissible, since $B \subseteq A$, hence $B_f \subseteq A_f$; in other words, $B_f \in \mathsf{Bis}(pf,p)$. Consider $\mathsf{Lift}(B_f) \in \mathcal{E}$. 
Applying $(-)_f$ to ($*$) and pushing it to the right, we obtain $M M^{\dagger} B_f = M_f$. However, we also have $(M M^{\dagger} \mathsf{Lift}(B_f))_f = M M^{\dagger} B_f$, since $(\mathsf{Lift}(B_f))_f = B_f$. From the uniqueness part of Proposition \ref{prop:deseqframes}.\ref{eq: second condition for descent}., since both $M M^{\dagger} \mathsf{Lift}(B_f)$ and $M$ belong to $\mathcal{E}$, it follows that $M M^{\dagger} \mathsf{Lift}(B_f) = M$. As a consequence, $\mathsf{Lift}(B_f) \subseteq M$. We can then conclude that $B = \mathsf{Lift}(B_f) \in \mathcal{E}$: they are both subsets of $M$, and both $B \to B_f$ and $\mathsf{Lift}(B_f) \to B_f$ are the restrictions of $M \to M_f$, which is bijective by Proposition \ref{prop:deseqframes}.\ref{eq: first condition for descent}., since $M \in \mathcal{E}$.

Closure under non-empty unions is now easy. Let $X, Y \in \mathcal{W}$ and let $A, B \in \mathsf{B}(X,Y)$; assume that $A, B \in \mathcal{E}$. We want to prove that $A \cup B \in \mathcal{E}$. Consider $M$ greatest element of $\mathsf{B}(X,Y) \cap \mathcal{E}$. Then, by $A \subseteq M$ and $B \subseteq M$, we have that $A \cup B \subseteq M$. Using downward closure, since $A \cup B \in \mathsf{B}(X,Y)$, we have that $A \cup B \in \mathcal{E}$.
\end{proof}

\begin{theorem}\label{cor: almost barr exactness}
The categories $\Fr_{l.f}$ and $\Fr_{f}$ are almost Barr-exact.
\end{theorem}
\begin{proof}
Recall that both $\Fr_{l.f}$ and $\Fr_{f}$ are regular categories (see e.g., \cite[Lemma 5.12.i)]{Ghilardi2002}). 
By Corollary \ref{cor:almostchar}, almost Barr exactness corresponds to regularity together with coequivalence relations arising from descent data being effective. Thus the result follows from Theorem \ref{thm:desimplieseff}.
\end{proof}

In \cite{deberardinisprooftheoryforprofinite}, to a normal modal logic $L$ (with the finite model property) is associated an infinitary extension $L^\infty$, whose algebraic models are given by profinite $L$-algebras (inverse limits of finite $L$-algebras). The category of profinite $L$-algebras (and morphisms preserving arbitrary meets and joins) is dual to the category $L\mathsf{Fr}_{l.f}$ of locally finite Kripke frames validating $L$ (and p-morphisms). The same bridge theorems connecting syntactic properties of a logic $L$ and categorical properties of $\mathsf{Alg}(L)_{fp}^{op}$ hold for $L^\infty$ and $L\mathsf{Fr}_{l.f}$.

\begin{theorem}\label{cor:S5LCoSP}
The logics $\mathsf{S5}$ and $\mathsf{S5}^\infty$ have the \emph{LCoSP}.
\end{theorem}
\begin{proof}
It is well known that  $\mathsf{S5}$  enjoys $\exists_{\mathrm{MIP}}$ (since it has Craig interpolation \cite{Maksimova1999}, and is locally tabular). Thus by Theorem \ref{thm: characterization of almost Barr-exactness}, for it to have LCoSP is equivalent for $\mathsf{Alg}(L)_{fp}^{op}$ to be almost Barr-exact. This is precisely what we showed in Theorem \ref{cor: almost barr exactness}.
\end{proof}

Theorem \ref{cor: almost barr exactness} suggests that modal and intuitionistic logics open a new class of interesting examples of almost Barr-exact categories, which are neither locally cartesian categories, nor Barr-exact, which might be attractive for the purposes of descent theory.

\section{Acknowledgements}

The authors would like to thank Silvio Ghilardi, Rui Prezado and Balder ten Cate for valuable discussions and pointers on the subject of this paper, and the anonymous referees for suggestions that improved the presentation of the results.

\bibliographystyle{eptcs}
\bibliography{generic}

\end{document}